\documentclass{article}

\usepackage[english]{babel}
\usepackage{amsfonts}
\usepackage{amsmath}
\usepackage{amsthm}
\usepackage{bm}
\usepackage{ulem}
\usepackage{graphicx,epstopdf} 
\usepackage[caption=false]{subfig} 
\usepackage{marvosym}

\newtheorem{theorem}{Theorem}[section]

\newtheorem{lemma}{Lemma}[section]

\newtheorem{definition}{Definition}[section]

\newtheorem{remark}{Remark}[section]

\def\brho{\overline{\rho}}
\def\bh{\overline{h}}
\def\bn{\overline{n}}

\def\p{\partial}
\usepackage[letterpaper,top=2cm,bottom=2cm,left=3cm,right=3cm,marginparwidth=1.75cm]{geometry}

\usepackage{amsmath}
\usepackage{graphicx}
\usepackage[colorlinks=true, allcolors=blue]{hyperref}
\usepackage{cleveref}
\author{Yaming Zhang 
\and Ning Jiang
\and Jiangyan Liang 
\and Yi-Long Luo
\and Min Tang}

\author{Yaming Zhang \thanks{School of Mathematics, Institute of Natural Sciences and MOE-LSC, Shanghai Jiao Tong Unviersity, Shanghai, China (zymsasj@sjtu.edu.cn, jylmath@sjtu.edu.cn, tangmin@sjtu.edu.cn)}
\and Ning Jiang\thanks{School of Mathematics and Statistics, Wuhan University, Wuhan, China (njiang@whu.edu.cn)}
\and Jiangyan Liang \footnotemark[1]
\and Yi-Long Luo \thanks{School of Mathematics, South China University of Technology, Guangzhou, China (luoylmath@scut.edu.cn)}
\and Min Tang \footnotemark[1]
}
\date{}
\begin{document}
\title{Pattern formation of a pathway-based diffusion model: linear stability analysis and an asymptotic preserving method}

\maketitle

\begin{abstract}
We investigate the linear stability analysis of a pathway-based diffusion model (PBDM), which characterizes the dynamics of the engineered $Escherichia$ $coli$ populations [X. Xue and C. Xue and M. Tang, $PLoS$ $Computational$ $Biology$, 14 (2018), pp. e1006178]. This stability analysis considers small perturbations of the density and chemical concentration around two non-trivial steady states, and the linearized equations are transformed into a generalized eigenvalue problem. By formal analysis, when the internal variable responds to the outside signal fast enough, the PBDM converges to an anisotropic diffusion model, for which the probability density distribution in the internal variable becomes a delta function. We introduce an asymptotic preserving (AP) scheme for the PBDM that converges to a stable limit scheme consistent with the anisotropic diffusion model. Further numerical simulations demonstrate the theoretical results of linear stability analysis, i.e., the pattern formation, and the convergence of the AP scheme.
\end{abstract}
\section{Introduction}\label{sec1}
There are a wide variety of regularly spaced patterns such as vertebrate segments, hair follicles, fish pigmentation, or animal coats \cite{Baker2009Waves, Held1992Models, Marrocco2010Models, Murray2002Mathematical, Painter1999Stripe, Turing1952The, Volkening2015Modelling, Wang2017A}. These patterns are the outcome of coordinated intracellular cell signaling, cell-cell communication, cell growth or cell migration. It is difficult to uncover the essential mechanisms for pattern formation that often are buried in extremely complex physiological contexts. In addition, synthetic biology for  bacteria or simple eukaryotes has been recently used to examine potential strategies for pattern formation \cite{Basu2005A, Khalil2010Synthetic, Mukherji2009Synthetic}.

Cells or organisms can bias their movements in response to extracellular chemical signals. This property is called chemotaxis. It often plays an essential role in innate immunity biofilm-associated infections, embryonic development, tissue maintenance or cancer metastasis \cite{Friedl2009Collective, O'Toole1999The, Pittman2001Chemotaxis, Singh2006Biofilms, Williams2007Helicobacter}. From the macroscopic point of view, continuum models of chemotaxis can characterize the evolution of the cell density by utilizing partial differential equations (PDEs). Among these PDEs models, the classical Patlak–Keller–Segel (PKS) chemotaxis equations have been a popular topic for decades \cite{Keller1970Initiation, Patlak1953Random}. On the other hand, from the microscopic view, run-and-tumble bacteria have been well studied. The intracellular chemotactic  signaling pathway as well as its relation to cell movement were investigated recently \cite{Si2012Pathway,Min2021MODELING, Xue2015Macroscopic}. In order to assemble information of both microscopic and macroscopic levels, in \cite{Erban2004From, Xue2009Multiscale}, the PKS equations were derived from the kinetic models with internal state for bacterial chemotaxis. Using the assumption that the methylation level concentrates locally, the authors proposed a new kinetic system with the turning operator that involves the dynamical intracellular pathway in \cite{Si2014A}. More macroscopic models and their derivations can be found in \cite{Perthame2020Multiple,Sun2017Macroscopic}.

The $Escherichia$ $coli$ chemotaxis signaling pathway was engineered in \cite{Liu2011Sequential}, which involved a quorum sensing module leading to cell motility suppressed by cell-density. The engineered cells in semi-solid agar form periodic stripes of high and low cell densities. Recently in \cite{Xue2018The}, a hybrid particle model that incorporates the intracellular signaling described in \cite{Xue2015Macroscopic} was developed, including a detailed description of intracellular signaling, single cell movement and cell division. This model can be used to illustrate the role of intracellular signaling in stripe formation and explain how the spatial stripe structure depends on cell-level parameters. Furthermore, the authors in \cite{Xue2018The} also derived a pathway-based diffusion model (PBDM) from this hybrid model by using moment closure method, which was consistent with the hybrid model in 1D numerical simulation. 

The PBDM derived in \cite{Xue2018The} reads as following:
\begin{align}\label{sys:kinetic}\tag{PBDM}
\begin{cases}
\partial_{t} \rho^{z} = \nabla_{\mathbf{x}}\cdot(D(z)\nabla_{\mathbf{x}} \rho^{z}) - \kappa \partial_{z} (g(z, h)\rho^{z}) + r n\rho^{z},\\
\partial_{t} h  = D_{h}\Delta_{\mathbf{x}} h +\alpha \varrho - \beta h,\\
\partial_{t} n  = D_{n}\Delta_{\mathbf{x}} n -\gamma \varrho  n,\\
\varrho(\mathbf{x},t) = \int_0^{Z_w}\rho^z\mathrm{d}z,
\end{cases}
\end{align}
with the initial data
\begin{align}\label{eq:initial data}
\rho ^z(\mathbf{x}, z, 0) =\rho^{z,0}(\mathbf{x},z),\ \ h(\mathbf{x}, 0) =h^0(\mathbf{x}),\ \ n(\mathbf{x}, 0) =n^0(\mathbf{x}),
\end{align}
where $\rho^{z,0}(\mathbf{x},z)$, $h^0(\mathbf{x})$, $n^0(\mathbf{x})$ are all non-negative functions. From the biological viewpoint, $\rho^z(\mathbf{x},z,t)$ is the density of engineered $Escherichia$ $coli$ cells at time $t \ge 0$, at position $\mathbf{x} \in \Omega \subseteq \mathbb{R}^N$ ($N \geq 1$) with internal state $z \in [0, Z_w]$, and $\varrho(\mathbf{x},t)$ stands for the total density. Specifically, $z$ is the total concentration of CheZ protein inside cells and $Z_w$ is the maximum amount of Chez protein. The scalar function $h(\mathbf{x}, t)$ is the acyl-homoserine lactone (AHL) concentration at time $t \ge 0$ and at position $\mathbf{x}\in \Omega$. $n(\mathbf{x}, t)$ is the local nutrient concentration at time $t \ge 0$ and at position $\mathbf{x} \in \Omega$. $\kappa > 0$ stands for a scaling parameter, which corresponds to the response speed of intracellular CheZ to the external signal AHL. $r > 0$ represents the growth rate of engineered $Escherichia$ $coli$ cells. The constants $D_h > 0$, $D_n > 0$ are the diffusion coefficients of AHL and the nutrient, respectively. $\alpha > 0$ denotes the production rate of AHL, $\beta > 0$ characterizes the degradation rate of AHL and $\gamma>0$ is the consumption rate of the nutrient.
The diffusion coefficient $D(z)$ is a monotonically increasing non-negative function of $z$ with $D(z)>0$. More specifically, $D(z)$ characterizes the fact that cells with different intracellular CheZ concentration $z$ have different mobility. The dynamics of the total concentration of CheZ protein $z(t)$ is governed by protein production due to transcription and translation as well as dilution due to cell growth. In order to
model this effect, the intracellular dynamic function $g (z,h)$ is given by the following form
\cite{Xue2018The}:
\begin{equation}\label{g(zh)}
  \begin{aligned}
    g (z, h) =k_V\tilde g(z,h)= k_V (L (h) - z) \,,
  \end{aligned}
\end{equation}
where $k_V$ is the volume growth rate, which might be only related to the cell growth rate $r$ ($k_V =r$) or a linear function of the local nutrient concentration $n(\mathbf{x},t)$ ($k_V =r n(\mathbf{x},t)$) \cite{Xue2018The}. The steady state $L(h)$ of the intracellular CheZ is a non-increasing smooth function of $h$, satisfying $L(0) = Z_w$ and $L(\infty) = 0$.
In \cite{Xue2018The}, since AHL suppresses the production of CheZ protein in an ultra-sensitive way, $L(h)$ is chosen to be a step function that is discontinuous at $h=h_0$, where $h_0$ is the threshold AHL level for the suppression of CheZ. In this paper, to simplify the problem, we smooth out the step function and consider the following form
\begin{equation}\label{eq:Lh}
L(h)=Z_w(0.5-0.5\tanh(\mu(h-h_0))),\end{equation} where $Z_w$ is the CheZ level of wild type $Escherichia$ $coli$ cells and $\mu$ gives the sensitivity of CheZ production to the AHL concentration. We impose no-flux boundary conditions of internal state $z$ and space variable $\mathbf{x}$ for any $t>0$:
\begin{equation}\label{BC}
   \begin{aligned}
    &\rho^z|_{z=0^-} =  \rho^z|_{z = Z_w^+} = 0,,\quad \mathbf{x} \in  \Omega\\
    &\nabla_{\mathbf{x}} \rho^{z} \cdot \mathbf{n} =0,\quad \mathbf{x} \in \partial \Omega,\quad z \in [0, Z_w],\\
    &\nabla_{\mathbf{x}} h \cdot \mathbf{n} =\nabla_{\mathbf{x}} n \cdot \mathbf{n} = 0,\quad \mathbf{x} \in \partial \Omega,
   \end{aligned}
\end{equation}
where $\mathbf{n}$ is outward normal vector. 

The engineered $Escherichia$ $coli$ cells secret AHL and AHL suppresses the production of CheZ protein. Thus when the cell density is high, the AHL concentration is high which induces lower intracellular CheZ protein concentration and the number of tumbling cells increases. In other words, cell motility reduces in the place with high cell density. The lower the motility, the harder the cells can escape from high-density regions, which leads to local cell aggregation and formation of sequential stripe patterns as has been observed in the experiment in \cite{Liu2011Sequential}. This phenomenon can be well explained by PBDM and one interesting question is whether there exist other patterns and under which conditions nontrivial patterns can appear.


In \cite{jiang2021kinetic}, by taking the limit $\kappa\to\infty$, i.e. the asymptotic behavior of the fast CheZ turnover rate, we  $formally$ get
$$
\rho(x,t,z)=\varrho(x,t)\delta(L(h(x,t))-z)
$$
and $\varrho(x,t)$ satisfies the anisotropic diffusion model (ADM) with the following form:
\begin{equation}\label{sys:limit}\tag{ADM}
\begin{aligned}
\begin{cases}
\partial _t\varrho(\mathbf{x},t) =\Delta_{\mathbf{x}}\left(D(L(h(\mathbf{x},t)))\varrho(\mathbf{x},t)\right) + r n\varrho(\mathbf{x},t),\\
\partial_{t} h(\mathbf{x}, t) =D_{h}\Delta_{\mathbf{x}} h(\mathbf{x}, t) +\alpha \varrho(\mathbf{x}, t)  -\beta h(\mathbf{x}, t),\\
\partial_{t} n(\mathbf{x}, t) =D_{n}\Delta_{\mathbf{x}} n(\mathbf{x}, t) -\gamma \varrho (\mathbf{x}, t)  n(\mathbf{x} t).
\end{cases}
\end{aligned}
\end{equation}
Compared with (\ref{sys:kinetic}), the cell mobility of (\ref{sys:limit}) is given by $D(L(h(\mathbf{x},t)))$, depending directly on concentration of AHL. The key feature of ADM is the term $\Delta_{\mathbf{x}}\left[D(L(h(\mathbf{x},t)))\varrho(\mathbf{x},t)\right]$ in the equation for $\varrho$, which is called anisotropic diffusion term in \cite{jiang2021kinetic}. Since $$\Delta_{\mathbf{x}}\left(D(L(h))\varrho(\mathbf{x},t)\right)=\nabla_{\mathbf{x}}\left(D(L))\nabla_{\mathbf{x}}\varrho(\mathbf{x},t)\right)+\nabla_{\mathbf{x}}\cdot\left(\nabla_{\mathbf{x}} D(L(h))\varrho(\mathbf{x},t)\right),$$$\Delta_{\mathbf{x}}\left[D(L(h(\mathbf{x},t)))\varrho(\mathbf{x},t)\right]$ includes not only a diffusion term but also an advection term. The anisotropic diffusion term is of particular interest since it appears in other models and applications in the literature. For example, the cross diffusion-reaction limit model of Shigesada-Kawasaki-Teramoto system \cite{DDD19,M06}, the stationary Fokker–Planck–Kolmogorov system \cite{B18}, the Kermack–McKendrick model with nonlocal source terms \cite{BOS22}, the predator-prey model with nonlinear cross-diffusion \cite{CSH21} and so on. 

There exist some analytical works for both PBDM and ADM in the literature. On the one hand the mathematical analysis of PBDM starts from our last paper \cite{jiang2021kinetic}. This system is a highly nonlinear coupling of reaction-diffusion and kinetic-type equations. The first step is to investigate the existence of smooth solutions near constant states, which are closely related to their stability. In \cite{jiang2021kinetic}, the global well-posedness around the trivial state $(0,0,0)$ is proved. On the other hand, \eqref{sys:limit} is of the same type as the  model introduced in \cite{Liu2011Sequential}, where the authors numerically reproduced some key features of experimental observations. Then the mechanism of this patterning process was studied in \cite{FTL12} based on a similar model as in \eqref{sys:limit}, whose asymptotic stability was studied later on in \cite{JSW20,MPW20}. 


The goal of this paper is to investigate the similarity and differences between these two models in terms of linear stability analysis and pattern formation, as well as provide a uniform convergent scheme with respect to the parameter $\kappa$.

Firstly, we investigate the stability of the model \eqref{sys:kinetic} around the non-zero constant states $(\brho^z, \bh, 0)$ and $(0,0,\bn)$. Specifically, we consider the stability of the linearized system \eqref{sys:kinetic} around the above non-trivial steady states with different forms of $k_V$. The main difficulty comes from the z-flux term $\kappa \p_z(g(z,h)\rho^z)$ since $\brho^z$ might be a $\delta$-function $\bar\varrho\delta(z - L(\bar h))$. 
The delta function distribution in $z$ is the main challenge of the analysis and computation. The classical stability analysis fails for the system \eqref{sys:kinetic}. One has to reformulate the equation first based on the structure of the solution. 
To overcome this difficulty, we integrate the cell density $\rho^z$ near the neighborhood of the singularity and obtain the following new system:

\begin{align*}
\begin{cases}
\partial_{t} \varrho_\theta = D(L(\bar h))\Delta_{\mathbf{x}}\varrho_\theta - \kappa \left[g(L^+,h)\rho^{L^+} - g(L^-,h)\rho^{L^-}\right]+r n\varrho_\theta,\\
\partial_{t} \rho^{z} =\nabla_{\mathbf{x}}\cdot(D(z)\nabla_{\mathbf{x}} \rho^{z}) - \kappa \partial_{z} (g(z, h)\rho^{z}) +r n\rho^{z}, \\
\partial_{t} h =D_{h}\Delta_{\mathbf{x}} h +\alpha[\varrho_\theta + \int_{\Omega_{\theta}}\rho^z\mathrm{d}z] -\beta h,\\
\partial_{t} n =D_{n}\Delta_{\mathbf{x}} n -\gamma [\varrho_\theta + \int_{\Omega_{\theta}}\rho^z\mathrm{d}z] n,
\end{cases}
\end{align*}
where $L^{\pm} = L(\bar h)\pm \theta$ with $\theta>0$ being a sufficiently small constant, $z\in \Omega_\theta:= [0,Z_w]/(L^-,L^+)$, $\varrho_\theta(\mathbf{x},t) = \int_{L^-}^{L^+}\rho^z\mathrm{d}z$ is the total density for $z\in[L^-, L^+]\cap [0,Z_w]$.

Then, an asymptotic preserving (AP) scheme that connects the two models \eqref{sys:kinetic} and \eqref{sys:limit} is designed. Due to the analytical difficulties, we only justify formally the limiting process ($\kappa\to\infty$) from the model \eqref{sys:kinetic} to the model \eqref{sys:limit} in \cite{jiang2021kinetic}. But the numerical results on fine meshes in this work indicate the validity of this convergence (see \Cref{fig2}). Then we design an AP scheme that has a uniform convergence order with respect to $\kappa$. A scheme is AP when the asymptotic limit of the discretization  becomes  a stable  solver for the limit model as the scaling parameter $\kappa\to\infty$. When $\kappa$ is large, to guarantee the scheme accuracy, one has to use mesh sizes that are smaller than $1/\kappa$. AP schemes allow for unresolved meshes and provide a general framework for solving the numerical difficulty due to some multiscale parameters \cite{Hu2017AsymptoticPreservingSF, Jin2010ASYMPTOTICP}. The main challenge is the limiting delta distribution in the internal variable. We show that the AP scheme for PBDM converges ($\kappa\to\infty$) to a limit scheme, which is consistent with and a stable discretization for the ADM \eqref{sys:limit}. 


 Several reaction-diffusion-advection type equations have been analyzed due to their wide applications in population dynamics \cite{cantrell2008approximating, jiang2021kinetic,  Ma2021Bifurcation,VidalHenriquez2017Convective, Zhao2015A}. The authors in \cite{VidalHenriquez2017Convective} investigated the convective instability and the local absolute instability of a reaction-diffusion-advection system. In \cite{Ma2021Bifurcation}, the authors considered a two-species reaction-diffusion-advection competition model with the no-flux boundary condition and used the linear stability analysis method to derive the existence as well as local stability of the trivial and semi-trivial steady-state solutions. However, there are very few studies on the stability analysis of the reaction-diffusion model with a convection term depending on the other variable, such as an internal state. Up to our knowledge, there is almost no literature that provides formal stability analysis for solutions with Dirac singularities. The difficulty of designing an AP scheme for PBDM is due to the limiting delta distribution in the internal variable. A similar strategy as the scheme in  \cite{Vauchelet2021Numerical} can be employed to get a limiting delta distribution in $z$. However, using the idea in \cite{Vauchelet2021Numerical}, the $\rho(x,y,z)$'s distribution in $z$ can only be concentrated exactly at the point $z=L(h(x,y))$, which yields an unstable limiting scheme for ADM. In order to obtain a stable discretization for ADM, we propose a new strategy of using several branches of solutions that allow the $z$ distribution of the discretized $\rho(x,y,z)$ to be concentrated at $z=L(h(x',y'))$ with $(x',y')$ being at the neighbourhood of $(x,y)$.


The rest of this paper is organized as follows: in the next section, we investigate the linear stability for the model \eqref{sys:kinetic} around the two non-trivial steady states $(\brho, \bh, 0)$ and $(0, 0, \bn)$. The results depend on the types of steady state and we discuss the corresponding stability/instability of the model \eqref{sys:kinetic}. Since the different forms of the volume growth rate $k_V$ affect the stability result, our analysis is divided into two cases: $k_V = r$  and $k_V = r n(\mathbf{x},t)$. In \Cref{sec3}, we design an asymptotic preserving scheme of model \eqref{sys:kinetic} with a scaling $\kappa\to\infty$. By using this scaling, an anisotropic diffusion model is $formally$ derived from model \eqref{sys:kinetic}. In \Cref{sec4}, some numerical simulations are presented to verify our analytical results and show several interesting patterns, like the circular sector. Finally, we make some discussion on the results of this paper and some possible future work.


\section{Linear stability analysis}\label{sec2}
We first determine the non-trivial steady solution $(\bar\rho^z,\bar h,\bar n)$ that satisfies
  \begin{align}\label{eq:constraint}
    \begin{cases}
     \Delta _{\mathbf{x}}\bar \rho^z =0,\quad\Delta _{\mathbf{x}}\bar h =0,\quad \Delta _{\mathbf{x}}\bar n =0,\\
      r\bar n\bar\rho^z=0,\quad \alpha \bar\varrho -\beta\bar h =0,\quad \gamma \bar \varrho \bar n =0,\\
      \partial_z(g(z,\bar h)\bar \rho^z) =0.
    \end{cases}
  \end{align}
There is always a trivial uniform steady state $(0,0,0)$ which was already studied in \cite{jiang2021kinetic}. In this paper, we focus on two types of non-zero steady states $(\bar\rho^z, \bar h, 0)$ and $(0, 0, \bar n)$. The main difficulty comes from that $\bar\rho^z$ depends on the internal state $z$. This is different with the classical reaction-diffusion-advection system.

Since the volume growth rate $k_V$ can depend on the cell growth rate $r$ ($k_V =r$) or be a linear function of the nutrient $n(\mathbf{x},t)$, 
we consider the following two cases:

\begin{itemize}
  \item [(1)]As $k_V = r$, $\bar\rho^z$ is a $\delta\mbox{-}$function in $z$. According to the constraints in \eqref{eq:constraint}, we obtain
    \[\partial_z\left(g(z,\bar h)\bar\rho^z\right) = 0\quad \Rightarrow \quad  r\partial_z\left((L(\bar h) - z)\bar\rho^z\right) = 0 \quad \Rightarrow \quad (L(\bar h) - z)\bar\rho^z = C,\]
where the constant $C$ is independent of $z$. If $\bar\rho^z$ is a continuous function of $z$, then  $\bar\rho^z = \frac{C}{L(\bar h)-z}$. due to $L(\bar h) \in [0,Z_w]$, we can obtain $\bar\varrho = \int^{Z_w}_0\bar\rho^z\mathrm{d}z = \infty$ , which contradicts the biological background. Then $\bar\rho^z$ is a $\delta\mbox{-}$function, i.e.,
\[\bar\rho^z = \bar\varrho\delta(z - L(\bar h)).\]
  \item [(2)] As $k_V = rn(x,t)$, we deduce that
  \begin{align}\label{rho-rn}
    \p_z\left(g(z,\bar h)\bar\rho^z\right) = r\bar n\partial_z\left((L(\bh) - z)\bar\rho^z\right) = 0.
  \end{align}
The distribution of $\bar\rho^z$ with respect to $z$ can not be determined since $\bar n$ is zero in the steady state $(\brho^z, \bh, 0)$. So we only consider that the state $\brho^z$ is a continuous function in $z$ and $\brho^z = \bar\varrho\delta(z - L(\bar h))$ in the subsequent part.
\end{itemize}

The classical way of formal linear stability analysis is to linearize the model \eqref{sys:kinetic} around the steady state $(\bar\rho^z,\bar h,\bar n)$ by introducing
\begin{equation*}
\begin{aligned}
&\rho^z =\bar\rho ^z+\delta \rho ^z,\quad h =\bar h +\delta h, \quad n=\bar n+\delta n.\\
\end{aligned}
\end{equation*}
Assume that the perturbations $(\delta\rho^z,\delta h,\delta n)$ are small and can be represented by harmonic waves as follows:
\begin{equation*}
\begin{aligned}
    \left(
  \begin{array}{l}
  \delta \rho^z\\
  \delta h\\
  \delta n
  \end{array}
  \right)= \sum_{\mathbf{k}}
      \left(
  \begin{array}{l}
           C_1^{\mathbf{k}} e^{i k_3z}\\
           C_2^{\mathbf{k}}\\
           C_3^{\mathbf{k}}
  \end{array}
  \right)e^{\lambda t + i k_1 x + i k_2 y},
\end{aligned}
\end{equation*}
where $\mathbf{k} = (k_1,k_2,k_3)^T \in \mathbb{R}^3$ and $k_1$, $k_2$, $k_3$ are the corresponding frequencies. 
Then the linearized system of \eqref{sys:kinetic} can be expressed in a matrix form
     \begin{equation}\label{de-eigenvalue problem}
    A \left\{\begin{array}{l}
     \delta \rho ^z \\
     \delta h \\
     \delta n
    \end{array}\right\} = \lambda
    \left\{\begin{array}{c}
     \delta \rho ^z \\
     \delta h \\
     \delta n
    \end{array}\right\},
    \end{equation}
where $A$ is the matrix of model \eqref{sys:kinetic} around the steady state and $\lambda$ is the corresponding eigenvalue. The eigenvalue problem \eqref{de-eigenvalue problem}
%
%
determines the linear stability of PBDM around the steady state as follows:
\begin{definition}
The model \eqref{sys:kinetic} is called stable at the steady state $(\bar\rho^z,\bar h,\bar n)$ if there exists a constant $c_0 > 0$ such that for all eigenvalues $\lambda$, we have $Re(\lambda) \leq - c_0$. Otherwise, it is unstable if there exists an eigenvalue $\lambda$ with $Re(\lambda) > 0$.
\end{definition}
The above approach can be applied to ADM (see \Cref{SM1} in the Appendix) but not to PBDM. When we consider the linearization of $g(z,h)$ around the steady state $(\bar\rho^z,\bar h, 0)$ with $k_V = r$, the advection term can be approximated by
    \begin{equation*}
        \begin{aligned}
        \kappa\partial_z\left(g(z,h)\rho^z\right) 
        = \kappa  \partial_z\left[g(z,\bar h)\delta \rho^z\right] + \kappa \partial_z\left[\partial_h g(z,\bh)\bar \rho^z\right]\delta h,
        \end{aligned}
    \end{equation*}
where the higher order terms are ignored. Note that $\partial_h g(z,\bar h)$ near $0$ as $\bh$ is away from $h_0$ but becomes nonzero as $\bar h$ approaches $h_0$, while $\bar \rho^z$ has a delta distribution in $z$. Then it is hard to control the term $\partial_z\left[\partial_h g(z,\bh)\bar \rho^z\right]$, which indicates that the perturbation of $\kappa \partial _z(g(z,h)\rho^z)$ near $(\brho^z, \bh, 0)$ is the main difficulty.

\subsection{Linearizations}\label{subsec:Preliminaries}

 In order to solve this problem, we define
\begin{equation*}
    \begin{aligned}
    &L^{\pm} : = L(\bar h) \pm \theta, \quad \Omega_\theta := [0,Z_w]/(L^-,L^+),\quad \varrho_\theta(\mathbf{x},t) = \int_{L^-}^{L^+}\rho^z\mathrm{d}z,
    \end{aligned}
\end{equation*}
where $\theta > 0$ is a small constant.

Integrating the $\rho^z$-equation of model \eqref{sys:kinetic} in $z$ over $(L^-,L^+)$, we derive that
\begin{equation*}
\begin{aligned}
\partial_{t} \varrho_\theta &=\Delta_{\mathbf{x}}\int_{L^-}^{L^+}D(z)\rho^z\mathbf{d}z - \kappa \left[g(L^+,h)\rho^{L^+} - g(L^-,h)\rho^{L^-}\right]+r n(\mathbf{x}, t)\varrho_\theta,\\
&\approx D(L(\bar h))\Delta_{\mathbf{x}}\varrho_\theta - \kappa \left[g(L^+,h)\rho^{L^+} - g(L^-,h)\rho^{L^-}\right]+r n(\mathbf{x}, t)\varrho_\theta,\\
\end{aligned}
\end{equation*}
where $\rho^{L^{\pm}} = \rho^z|_{z = L^{\pm}}$. The approximation is valid since $\theta$ is small enough and $D(z)$ is continuous. Then model \eqref{sys:kinetic} can be rewritten into the following form:
\begin{align}\label{a14}
\begin{cases}
\partial_{t} \varrho_\theta = D(L(\bar h))\Delta_{\mathbf{x}}\varrho_\theta - \kappa \left[g(L^+,h)\rho^{L^+} - g(L^-,h)\rho^{L^-}\right]+r n\varrho_\theta,\\
\partial_{t} \rho^{z} =\nabla_{\mathbf{x}}\cdot(D(z)\nabla_{\mathbf{x}} \rho^{z}) - \kappa \partial_{z} (g(z, h)\rho^{z}) +r n\rho^{z}, \\
\partial_{t} h =D_{h}\Delta_{\mathbf{x}} h +\alpha[\varrho_\theta + \int_{\Omega_{\theta}}\rho^z\mathrm{d}z] -\beta h,\\
\partial_{t} n =D_{n}\Delta_{\mathbf{x}} n -\gamma [\varrho_\theta + \int_{\Omega_{\theta}}\rho^z\mathrm{d}z] n,
\end{cases}
\end{align}
where $ z\in \Omega_\theta$ in the second equation. From the definition of $\varrho_\theta$, we know that the boundary condition for $\varrho_\theta$ is the same as in \eqref{BC} and its initial data satisfy
\[\varrho_\theta^0 = \int_{L-}^{L^+}\rho^0(\mathbf{x},z)\mathbf{d}z.\]
We can obtain the stability of model \eqref{sys:kinetic} from the stability of system (\ref{a14}). The two steady states of \eqref{a14} are $(\bar\varrho_{\theta},\bar\rho^z,\bh,0)$ and $(0,0,0,\bn)$ with $\bar\varrho_\theta = \int_{L^-}^{L^+}\bar\rho^z\mathrm{d}z$.
Then we linearize \eqref{a14} around the steady state $(\bar\varrho_{\theta},\bar\rho^z,\bar h,\bar n)$. Let the solution $(\varrho_\theta, \rho^z, h, n)$ be rewritten into the following perturbation form:
\begin{equation}\label{a15}
    \begin{aligned}
    &\varrho_\theta = \bar\varrho_\theta + \delta \varrho_\theta,\quad \rho^z =\bar\rho ^z+\delta \rho ^z\ (z\in\Omega_{\theta}),\quad h =\bar h +\delta h, \quad n=\bar n+\delta n,\\
    \end{aligned}
\end{equation}
where the perturbation  $|\delta \bm{\phi} |\ll 1$ with $\bm{\phi} = (\varrho_{\theta},\rho^z, h,n)^T $. The initial condition is given by
  \begin{align*}
     (\varrho_{\theta}^0(\mathbf{x}),\,  \rho^0(\mathbf{x},z),\, h^0(\mathbf{x}),\, n^0(\mathbf{x})) = (\bar\varrho_{\theta}^0 + \delta\varrho_{\theta},\, \brho^{z,0} + \delta \rho^z,\, \bh^0 + \delta h,\, \bn^0 + \delta n).
  \end{align*}

Using \eqref{eq:constraint} and the perturbation form in \eqref{a15}, \eqref{a14} gives the following perturbation system:
\begin{small}
\begin{equation}\label{sys:linearization}
\begin{aligned}
\begin{cases}
\partial_t\delta\varrho_\theta =& D(L(\bh))\Delta_{\mathbf{x}}\delta \varrho_\theta - \kappa k_V \left[\tilde g(L^+,\bh)\delta\rho^{L^+} - \tilde g(L^-,\bh)\delta \rho^{L^-}\right] \\
&-\kappa k_V\left[\brho^{L^+} - \brho^{L^-}\right] \partial_h\tilde g(z,\bh)\delta h + r \bar \varrho_{\theta}\delta n + r\bar n\delta\varrho_{\theta},\\
\partial_{t} \delta \rho ^z =& D(z)\Delta _{\mathbf{x}}\delta\rho^z - \kappa k_V \partial_z(\tilde g(z,\bar h)\delta \rho^z) - \kappa k_V \partial_z\bar\rho^z\partial_h\tilde g(z,\bh)\delta h +r\bar \rho ^z \delta n + r\bar n\delta\rho^z\\
\partial_{t} \delta h  =& D_{h}\Delta _{\mathbf{x}}\delta h +\alpha\delta\varrho_\theta + \alpha \int_{\Omega_{\theta}}\delta \rho^z\mathrm{d}z  -\beta \delta h,\\
\partial_{t} \delta n = & D_{n}\Delta _{\mathbf{x}}\delta n -\gamma [\bar\varrho_{\theta} + \int_{\Omega_{\theta}}\bar\rho^z\mathrm{d}z]\delta n -\gamma \bn\delta \varrho_{\theta} -\gamma \bn\delta \rho^z,\\
\end{cases}
\end{aligned}
\end{equation}
\end{small}
with $z\in\Omega_{\theta}$ and $\partial_z\tilde g(z,\bar h) = -1$. Note that $\partial_h\tilde g(z,\bh) = \partial_h L(h)|_{h = \bh}$ is independent of $z$. Assume the perturbation $(\delta\varrho_{\theta},\delta\rho^z,\delta h,\delta n)$ has the following form of wave
\begin{equation}\label{perturbation1}
    \begin{aligned}
     \left(
  \begin{array}{l}
  \delta \varrho_\theta\\
  \delta \rho^z\\
  \delta h\\
  \delta n
  \end{array}
  \right)= \sum_{{\mathbf{k}}}
      \left(
  \begin{array}{l}
           C_0^{\mathbf{k}}\\
           C_1^{\mathbf{k}} e^{i k_3z}\\
           C_2^{\mathbf{k}}\\
           C_3^{\mathbf{k}}
  \end{array}
  \right)e^{\lambda t + i k_1 x + i k_2 y},
    \end{aligned}
\end{equation}
where ${\mathbf{k}} = (k_1,k_2,k_3)^T\in \mathbb{R}^3$ and $k_1$, $k_2$, $k_3$ are corresponding frequencies. Inserting the perturbation \eqref{perturbation1} into the system \eqref{sys:linearization}, the linearized system can be expressed in a matrix form
     \begin{equation}\label{eigenvalue problem}
    A \left\{\begin{array}{l}
     C_{0}^{\mathbf{k}}\\
     C_{1}^{\mathbf{k}}\\
     C_{2}^{\mathbf{k}}\\
     C_{3}^{\mathbf{k}}
    \end{array}\right\} = \lambda
    \left\{\begin{array}{c}
     C_{0}^{\mathbf{k}}\\
     C_{1}^{\mathbf{k}}\\
     C_{2}^{\mathbf{k}}\\
     C_{3}^{\mathbf{k}}
    \end{array}\right\},
    \end{equation}
where $A$ is the linearized matrix operator of \eqref{sys:linearization}. For the two different steady states and the different forms of $k_V$, we respectively derive the corresponding matrix $A$ and determine the stability according to the signs of real parts of their eigenvalues.

We have the following theorem:
\begin{theorem}[Linear stability analysis]\label{Thm:linear stability}
 Let $\Omega\subset \mathbf{R}^2$ be a bounded domain, the stability of the model \eqref{sys:kinetic} is related to the cell volume growth rate $k_V$.  
 
\begin{itemize}
  \begin{item} $\bm{Case A:}$ For $k_V = r$, the model \eqref{sys:kinetic} is unstable at both steady states $(\brho^z, \bh, 0)$ and $(0,0,\bar n)$.
  \end{item}  

  \begin{item} $\bm{Case B:}$ For $k_V = r n(\mathbf{x},t)$, the model \eqref{sys:kinetic} is stable at $(\bar\rho^z,\bar h,0)$, but is unstable at $(0,0,\bar n)$.
  \end{item}
\end{itemize}
\end{theorem}

\subsection{The stability analysis when $k_V = r$}\label{subsec: k_V=r}
In this subsection, our aim is to investigate the stability of model \eqref{sys:kinetic} with $k_V = r$. 
Then \eqref{sys:linearization} becomes
\begin{equation}\label{sys:perturbation1}
\begin{aligned}
\begin{cases}
\partial_t\delta\varrho_\theta =& D(L(\bh))\Delta_{\mathbf{x}}\delta \varrho_\theta
 - \kappa r \left[
  \tilde g(L^+,\bh)\delta\rho^{L^+}
 - \tilde g(L^-,\bh)\delta \rho^{L^-}
 \right]\\
&- \kappa r\left[\brho^{L^+} - \brho^{L^-}\right] \partial_h\tilde g(z,\bh)\delta h + r\bar n\delta\varrho_{\theta}+ r \bar \varrho_{\theta}\delta n\\
\partial_{t} \delta \rho ^z =& D(z)\Delta _{\mathbf{x}}\delta\rho^z - \kappa r \partial_z(\tilde g(z,\bar h)\delta \rho^z) - \kappa r \partial_z\bar\rho^z\partial_h\tilde g(z,\bh)\delta h + r\bar n\delta\rho^z+r\bar \rho ^z \delta n\\
\partial_{t} \delta h  =& D_{h}\Delta _{\mathbf{x}}\delta h +\alpha\delta\varrho_\theta + \alpha \int_{\Omega_{\theta}}\delta \rho^z\mathrm{d}z  -\beta \delta h,\\
\partial_{t} \delta n  =& D_{n}\Delta _{\mathbf{x}}\delta n  -\gamma [\bar\varrho_{\theta} + \int_{\Omega_{\theta}}\bar\rho^z\mathrm{d}z]\delta n -\gamma \bn\delta \varrho_{\theta} - \gamma \bn\int_{\Omega_{\theta}}\delta \rho^z\mathrm{d}z,\\
\end{cases}
\end{aligned}
\end{equation}
where $z\in\Omega_\theta$ and $\partial_z\tilde g(z,\bar h) = -1$. Substituting the perturbation \eqref{perturbation1} into the perturbation system \eqref{sys:perturbation1} yields the characteristic matrix as follows:
\begin{equation*}
\begin{aligned}
    \left(
  \begin{array}{cccc}
          -K D(L(\bh)) + r\bn - \lambda &a_{12} & a_{13} &r\bar \varrho_{\theta}  \\
          0& a_{22} - \lambda& \kappa r\partial_z\brho^z\partial_h\tilde g(z,\bh) &r\brho^z\\
          \alpha & \alpha\int_{\Omega_{\theta}}e^{i k_3 z}\mathrm{d}z & -K D_h - \beta - \lambda &0\\
          -\gamma\bn &-\gamma\bn\int_{\Omega_{\theta}}e^{i k_3z}\mathrm{d}z&0& a_{44} -\lambda
  \end{array}
  \right),
\end{aligned}
\end{equation*}
where $K = k_1^2 + k_2^2$ and
\begin{equation*}
\begin{aligned}
a_{12} &= -\kappa r \left[\tilde g(L^+,\bh)e^{i k_3 L^+} - \tilde g(L^-,\bh)e^{i k_3 L^-}\right],\quad a_{13} = \kappa r\left[\brho^{L^+} - \brho^{L^-}\right] \partial_h\tilde g(z,\bh),\\ a_{22} &= -K  D(z) + \kappa r + r\bn - i k_3\kappa r \tilde g(z,\bh),\quad a_{44}=-K D_n - \gamma[\bar\varrho_{\theta} + \int_{\Omega_{\theta}}\bar\rho^z\mathrm{d}z].
\end{aligned}
\end{equation*}
\begin{itemize}
\item[$\bm{ A1:}$] \textbf{The steady state $(\bar\rho^z, \bar h, 0)$. }The corresponding steady state of \eqref{a14} is $(\bar\varrho_{\theta},\brho^z,\bh,0)$. Since $\brho^z = 0$ in $\Omega_{\theta}$, we have $a_{13} = 0$. Then eigenvalues are
\begin{align*}
&\lambda_{1}^1 = -K D(L(\bh)),\quad \lambda_{1}^2 = -K D(z) + \kappa r - i k_3\kappa r \tilde g(z,\bh),\\
&\lambda_{1}^3 = -K D_h - \beta,\quad\,\,\, \lambda_{1}^4 = -K D_n - \gamma\bar\varrho.
\end{align*}

\item[$\bm{ A2:}$] \textbf{The steady state $(0, 0, \bar n)$. }The corresponding steady state of \eqref{a14} is $(0,0,0,\bar n)$. The corresponding eigenvalues are
\begin{align*}
\lambda_{2}^1 &= -K D(L(\bh)) + r\bn,\quad \lambda_{2}^2 = -K  D(z) + \kappa r + r\bn - i k_3\kappa r \tilde g(z,\bh),\\
\lambda_{2}^3 &= -K D_h - \beta,\quad\quad\quad\,\,\,\, \lambda_{2}^4 = -K D_n.
\end{align*}
\end{itemize}
Observe that the real parts of $\lambda_{1}^2$ , $\lambda_{2}^1$ and $\lambda_{2}^2$ are positive if $K$ is sufficiently small. Therefore, it is concluded that, at both $(\bar\varrho_{\theta},\brho^z,\bar h,0)$ and $(0,0,0,\bar n)$, the linearized systems \eqref{sys:perturbation1} are unstable, i.e. \textbf{CaseA} of \Cref{Thm:linear stability} is proved.

\subsection{The stability analysis when $k_V = r n(\mathbf{x},t)$}\label{subsec: k_V=rn}
Similarly, the perturbation system of model \eqref{sys:kinetic} with $k_V = r n(\mathbf{x},t)$ is
\begin{equation}\label{sys:perturbation2}
\begin{aligned}
\begin{cases}
\partial_t\delta\varrho_\theta =& D(L(\bh))\Delta_{\mathbf{x}}\delta \varrho_\theta - \kappa r\bn \left[\tilde g(L^+,\bh)\delta\rho^{L^+} - \tilde g(L^-,\bh)\delta \rho^{L^-}\right] \\
&-\kappa r\bar n\left[\brho^{L^+} - \brho^{L^-}\right] \partial_h\tilde g(z,\bh)\delta h -\kappa r\left[\tilde g(L^+,\bh)\bar\rho^{L^+} - \tilde g(L^-,\bh)\bar \rho^{L^-}\right]\delta n
\\
&+ r\bar n\delta\varrho_{\theta} + r \bar \varrho_{\theta}\delta n,\\
\partial_{t} \delta \rho ^z =& D(z)\Delta _{\mathbf{x}}\delta\rho^z - \kappa r \bn\partial_z(\tilde g(z,\bar h)\delta \rho^z) - \kappa r\bn \partial_z\bar\rho^z\partial_h\tilde g(z,\bh)\delta h + r\bar n\delta\rho^z+r\bar \rho ^z \delta n \\
&- r\partial_z(g(z,\bh)\brho^z)\delta n,\\
\partial_{t} \delta h =&D_{h}\Delta _{\mathbf{x}}\delta h +\alpha\delta\varrho_\theta + \alpha \int_{\Omega_{\theta}}\delta \rho^z\mathrm{d}z  -\beta \delta h,\\
\partial_{t} \delta n =&D_{n}\Delta _{\mathbf{x}}\delta n -\gamma [\bar\varrho_{\theta} + \int_{\Omega_{\theta}}\bar\rho^z\mathrm{d}z]\delta n -\gamma \bn\delta \varrho_{\theta} -\gamma \bn\int_{\Omega_{\theta}}\delta \rho^z\mathrm{d}z,\\
\end{cases}
\end{aligned}
\end{equation}
where $z\in\Omega_\theta$ and $\partial_z\tilde g(z,\bar h) = -1$. Substituting the perturbation \eqref{perturbation1} into the system \eqref{sys:perturbation2} yields the characteristic matrix as follows:
\begin{equation*}
\begin{aligned}
    \left(
  \begin{array}{cccc}
          -K D(L(\bh)) + r\bn - \lambda &a_{12} & a_{13} &a_{14}  \\
          0& a_{22} - \lambda& \kappa r\bn\partial_z\brho^z\partial_h\tilde g(z,\bh) &a_{24}\\
          \alpha & \alpha\int_{\Omega_{\theta}}e^{i k_3 z}\mathrm{d}z & -K D_h - \beta - \lambda &0\\
          -\gamma\bn &-\gamma\bn\int_{\Omega_{\theta}}e^{i k_3z}\mathrm{d}z&0&a_{44} -\lambda
  \end{array}
  \right),
\end{aligned}
\end{equation*}
where $K = k_1^2 + k_2^2$ and
\begin{equation*}
\begin{aligned}
a_{12} &= -\kappa r \bn \left[\tilde g(L^+,\bh)e^{i k_3 L^+} - g(L^-,\bh)e^{i k_3 L^-}\right],\quad a_{13} = \kappa r\bar n\left[\brho^{L^+} - \brho^{L^-}\right] \partial_h\tilde g(z,\bh),\\
a_{14} & = r\bar \varrho_{\theta} - \kappa r\left[\tilde g(L^+,\bh)\bar\rho^{L^+} - \tilde g(L^-,\bh)\bar \rho^{L^-}\right],
\quad a_{24} = r\brho^z-\kappa r\partial_z(g(z,\bh)\brho^z),\\
a_{22} &= -K  D(z) + \kappa r\bn + r\bn - i k_3\bn \kappa r \tilde g(z,\bh),\quad a_{44}=-K D_n - \gamma[\bar\varrho_{\theta} + \int_{\Omega_{\theta}}\bar\rho^z\mathrm{d}z],\quad 
\end{aligned}
\end{equation*}

\begin{itemize}
\item[$\bm{B1:}$] \textbf{The steady state $(\bar\rho^z, \bar h, 0)$.} The corresponding steady state of \eqref{a14} is $(\bar\varrho_{\theta},\brho^z,\bar h,0)$. We obtain the following eigenvalues:
\begin{align*}
\lambda_{3}^1 = -K D(L(\bh)),\,\, \lambda_{3}^2 = -K  D(z),\,\, \lambda_{3}^3 = -K D_h - \beta,\,\, \lambda_{3}^4 = -K D_n - \gamma\bar\varrho.
\end{align*}
\item[$\bm{B2:}$] \textbf{The steady state $(0, 0, \bar n)$.} The corresponding steady state of \eqref{a14} is $(0,0,0,\bar n)$. The corresponding eigenvalues are:
\begin{align*}
\lambda_{4}^1 &= -K D(L(\bh)) + r\bn,\quad \lambda_{4}^2 = -K  D(z) + \kappa r + r\bn - i k_3\kappa r \tilde g(z,\bh),\\
\lambda_{4}^3 &= -K D_h - \beta,\quad\qquad \,\,\,\, \lambda_{4}^4 = -K D_n.
\end{align*}
\end{itemize}

It is noted that $\lambda_{3}^i$ ($i = 1,2,3,4$) are negative real values and the real parts of $\lambda_{4}^1$ as well as $\lambda_{4}^2$ are positive if $K$ is sufficiently small. Then we conclude that the linearized system \eqref{sys:perturbation2} is stable at the steady state $(\bar\varrho_{\theta},\brho^z,\bar h,0)$, but is unstable at the steady state $(0,0,0,\bar n)$ in this case, i.e., \textbf{CaseB} of \Cref{Thm:linear stability} is established.


\section{An asymptotic preserving scheme for \eqref{sys:kinetic}}\label{sec3}
In this part, we will give an AP scheme for PBDM that converges to a stable discretization for the limit model ADM when $\kappa\to \infty$.

When $k_V = r$, assume that $(\rho^{z,\kappa}(\mathbf{x},z,t),h^{\kappa}(\mathbf{x},t),n^{\kappa}(\mathbf{x},t))$ is a solution to the model \eqref{sys:kinetic} with the initial data \eqref{eq:initial data} and boundary condition \eqref{BC}. When $\kappa\to \infty$, we $formally$ obtain
\[\partial_{z} (g(z, h)\rho^{z}) = 0,\] 
which means that $\rho^z$ is a Dirac delta function in $z$ concentrating at $z = L(h(\mathbf{x},t))$ due to \eqref{g(zh)}. 
That $\rho^z(\mathbf{x},z,t)$ has delta distribution in $z$ brings some difficulties in designing the numerical scheme for PBDM. Let us first introduce some notations. The computational domain is chosen to be
\begin{align*}
\Lambda = \{(x,y,z)|(x,y,z)\in[-L_x,L_x]\times [-L_y,L_y]\times[0,Z_w]\},
\end{align*}
where $L_x$, $L_y$ and $Z_w$ are positive constants. In particular, the value of $Z_w$ is related to the biological experiment. Denote the uniform mesh sizes for $x$, $y$, $z$ and $t$ respectively by $\Delta x$, $\Delta y$, $\Delta z$ and $\Delta t$ and let
\begin{equation*}
\begin{aligned}
x_i = -L_x + i\Delta x,\quad y_j = -L_y + j\Delta y, \quad z_k = k\Delta z,\quad t^m = m\Delta t,
\end{aligned}
\end{equation*}
where $i\in \{0,1, 2,\cdots,N_{x}\}$, $j\in \{0,1, 2,\cdots,N_{y}\}$, $k\in \{0,1, 2,\cdots,N_z\}$ and $m\in N^+$, with $N_{x} = 2L_x/\Delta x$, $N_y = 2L_y/\Delta y$ and $N_z =Z_w/\Delta z$. We define the following approximations:
\begin{equation*}
\begin{aligned}
\rho^m_{i,j,k} &\approx \rho^z(x_i,y_j,z_k,t^m),\quad h^m_{i,j}\approx h(x_i,y_j,t^m),\\
n^m_{i,j}&\approx n(x_i,y_j,t^m),\quad\quad\,\, g_{i,j,k}^m \approx g(z_k,h_{i,j}^m).
\end{aligned}
\end{equation*}
The operators $\delta_{xx }(\cdot)$ and $\delta_{y y}(\cdot)$ are defined as follows:
    \begin{equation*}
        \begin{aligned}
        \delta _{xx}(u_{i,j}) = \frac{u_{i-1,j} - 2u_{i,j} + u_{i+1,j}}{\Delta x^2} ,\quad \delta _{yy}(u_{i,j}) = \frac{u_{i,j-1} - 2u_{i,j} + u_{i,j+1}}{\Delta y^2}.\\
        \end{aligned}
    \end{equation*}

Next we use the alternating difference implicit (ADI) method to reduce the computational cost. From $t^{m}$ to $t^{m+1}$, the AHL concentration $h(\mathbf{x},t)$ is updated by
\begin{equation}\label{dis_AHL}
\begin{aligned}
\frac {h^{m*}_{i,j}-h^m_{i,j}}{\Delta t/2}=&D_h[\delta _{xx}(h_{i,j}^{m*}) + \delta _{yy}(h_{i,j}^{m})]+\alpha \sum_{k=1}^{N_z} \rho _{i,j,k}^m\Delta z-\beta h^{m*}_{i,j},\\
\frac {h^{m+1}_{i,j}-h^{m*}_{i,j}}{\Delta t/2}=&D_h[\delta _{xx}(h_{i,j}^{m*}) + \delta _{yy}(h_{i,j}^{m+1})] + \alpha\sum_{k=1}^{N_z} \rho _{i,j,k}^m\Delta z-\beta h^{m+1}_{i,j}.\\
\end{aligned}
\end{equation}
The nutrient $n(\mathbf{x},t)$ is solved in the same way such that
\begin{equation}\label{dis_nutr}
\begin{aligned}
\frac {n^{m*}_{i,j}-n^m_{i,j}}{\Delta t/2}=&D_n[\delta _{xx}(n_{i,j}^{m*}) + \delta _{yy}(n_{i,j}^{m})] - \gamma n^{m*}_{i,j}\sum_{k=1}^{N_z} \rho _{i,j,k}^m\Delta z,\\
\frac {n^{m+1}_{i,j}-n^{m*}_{i,j}}{\Delta t/2}=&D_n[\delta _{xx}(n_{i,j}^{m*}) + \delta _{yy}(n_{i,j}^{m+1})] - \gamma n^{m+1}_{i,j}\sum_{k=1}^{N_z} \rho _{i,j,k}^m\Delta z.\\
\end{aligned}
\end{equation}

The most difficult part is the discretization of the $\rho^z$ equation. In \cite{Vauchelet2021Numerical}, an AP scheme for a kinetic equation with the internal state was proposed, where limiting delta concentration of the internal variable was considered as well. Straight forward extension of the idea in \cite{Vauchelet2021Numerical} gives the following limiting centered finite difference discretization of the $\Delta_{\mathbf{x}}\big(D(L(h))\varrho\big)$ term in ADM, such that
$$
\delta_{xx}\big(D(L(h_{i,j})\big)\varrho_{i,j})+\delta_{yy}\big(D(L(h_{i,j}))\varrho_{i,j}\big).
$$
However, due to the specific form of $L(h)$ in \eqref{eq:Lh}, $D(L(h))$ may have fast transition in space, the advection part can not be ignored in $\Delta_x\big(D(L(h))\varrho\big)$ and the above centered finite difference discretization is unstable. Therefore, we have to first prepare a stable discretization for ADM and then design the scheme for PBDM accordingly.

We discretize the $\rho^z$-equation of PBDM using time splitting method:
\begin{itemize}
    \item The first step is to solve the equation
    \begin{equation}\label{eq:timesplitting}
        \begin{aligned}
        \partial_t\rho^z + \kappa \partial_{z} (g(z, h)\rho^{z}) = 0.
        \end{aligned}
    \end{equation}
    for one time step.
We use the implicit upwind scheme to discretize this equation
    \begin{equation}\label{dis_split1}
        \begin{aligned}
        &\rho^{m*, \xi}_{i,j,k} = \rho^{m}_{i,j,k} - \tfrac{\kappa \Delta t}{\Delta z}(J^{m*, \xi}_{i,j,k+\frac{1}{2}} - J^{m*, \xi}_{i,j,k-\frac{1}{2}}),\\
        &J^{m*, \xi}_{i,j,k+\frac{1}{2}} = (g^{m + 1, \xi}_{i,j,k})^+\rho^{m*, \xi}_{i,j,k} - (g^{m + 1, \xi}_{i,j,k+1})^-\rho^{m*, \xi}_{i,j,k+1},\quad \text{for $k = 0,\cdots,N_z-1$,}\\
         &J^{m*, \xi}_{i,j,-\frac{1}{2}} = (g^{m + 1, \xi}_{i,j,-1})^+\rho^{m*, \xi}_{i,j,-1} \mathbf{1}_{\{h^{m+1}_{i,j}\leq h_0\}},\\
         &J^{m*, \xi}_{i,j,N_z+\frac{1}{2}} = - (g^{m + 1, \xi}_{i,j,N_z+1})^-\rho^{m*, \xi}_{i,j,N_z+1}\mathbf{1}_{\{h^{m+1}_{i,j}> h_0\}},\\
        \end{aligned}
    \end{equation}
    with 
    \begin{equation}\label{dis_gzh}
    \begin{aligned}
    &g_{i,j,k}^{m+1, \xi} = r\left(\mathcal{R}\left(\tfrac{L(h^{m+1,\xi}_{i,j})}{\Delta z}\right)\Delta z - z_k\right) : = r\left( L^{m+1,\xi}_{i,j} - z_k \right),
    \end{aligned}
    \end{equation}
    for $ \xi\in\{l,\,r,\,o,\,b,\,t\}$. Here 
    \[h^{m+1,\xi}_{i,j} = h^{m+1}_{i-1,j}\mathbf{1}_{\xi = l} + h^{m+1}_{i+1,j}\mathbf{1}_{\xi = r} + h^{m+1}_{i,j}\mathbf{1}_{\xi = o} + h^{m+1}_{i,j-1}\mathbf{1}_{\xi = b} + h^{m+1}_{i,j+1}\mathbf{1}_{\xi = t},\]
with $\mathbf{1}$ being the characteristic function; $\mathcal{R}(\cdot)$ is the rounding operator and $u^+ = \max\{0,u\}$, $u^- = \max\{0,-u\}$. We impose the no-flux boundary condition of internal variable $z$ such that
    \begin{equation}\label{dis_z_boundary}
       \begin{aligned}
     \rho^{m*, \xi}_{i,j,-1} = \rho^{m*, \xi}_{i,j,N_z+1} = 0,
       \end{aligned}
   \end{equation}
 where $\rho^{m*, \xi}_{i,j,-1}$ and $\rho^{m*, \xi}_{i,j,N_z+1}$ are the ghost points. The boundary condition \eqref{dis_z_boundary} yields $J^{m*,\xi}_{i,j,-\frac{1}{2}} = J^{m*,\xi}_{i,j,N_z+\frac{1}{2}} = 0$ for any $i$, $j$. It is important to note that different $\xi$ gives different values of $g_{i,j,k}^{m+1,\xi}$, which corresponds to solve not only  \eqref{eq:timesplitting}, but also
 $$
 \partial_t\rho^z(x,y,z,t) + \kappa \partial_{z} \left(g\Big(z, h(x\pm\Delta x,y\pm\Delta y,z,t)\Big)\rho^{z}(x,y,z,t)\right) = 0.
 $$
 For different $\xi$, $\rho^{m*, \xi}_{i,j,k}$ approximate each other when $\kappa=O(1)$, but concentrate at different $z$ when $\kappa\gg 1$.
    \item
    Secondly, we discretize
    \begin{equation*}
        \begin{aligned}
        \partial_t \rho^z = \nabla_{\mathbf{x}}\cdot(D(z)\nabla_{\mathbf{x}} \rho^{z}) + r n\rho^{z}.
        \end{aligned}
    \end{equation*}for one time step.
The discretization of this equation writes
    \begin{equation}\label{dis_split2}
        \begin{aligned}
        &\frac {\rho^{m+1}_{i,j,k}-\rho^{m*, o}_{i,j,k}}{\Delta t}=D(z_k)\sum_{p = 1}^5A_{{p,i,j,k}}^{m*} + r n^{m+1}_{i,j}\rho^{m*, o}_{i,j,k},\\
        \end{aligned}
    \end{equation}
    \begin{equation*}
        \begin{aligned}
        &A_{1,i,j,k}^{m*} = \tfrac{\rho^{m*,r}_{i-1,j,k} + \rho^{m*,o}_{i-1,j,k}}{2\Delta x^2}- \tfrac{(\rho^{m*,r}_{i-1,j,k} - \rho^{m*,o}_{i-1,j,k})\mathbf{1}_{\{h^{m+1}_{i-1,j}\leq h^{m+1}_{i,j}\}}}{\Delta x^2} 
        \\
        &A_{2,i,j,k}^{m*} = \tfrac{\rho^{m*,t}_{i,j-1,k} + \rho^{m*,o}_{i,j-1,k}}{2\Delta y^2}- \tfrac{(\rho^{m*,t}_{i,j-1,k} - \rho^{m*,o}_{i,j-1,k})\mathbf{1}_{\{h^{m+1}_{i,j-1}\leq h^{m+1}_{i,j}\}}}{\Delta y^2} 
        \\
        &A_{3,i,j,k}^{m*} = -\tfrac{\rho^{m*,l}_{i,j,k} +  2\rho^{m*,o}_{i,j,k} + \rho^{m*,r}_{i,j,k}}{2\Delta x^2} -\tfrac{ \rho^{m*,b}_{i,j,k}+ 2\rho^{m*,o}_{i,j,k}  + \rho^{m*,t}_{i,j,k}}{2\Delta y^2}\\
        &+\tfrac{(\rho^{m*,r}_{i,j,k} - \rho^{m*,o}_{i,j,k})\mathbf{1}_{\{h^{m+1}_{i,j}\leq h^{m+1}_{i+1,j}\}}}{\Delta x^2} +\tfrac{(\rho^{m*,t}_{i,j,k} - \rho^{m*,o}_{i,j,k})\mathbf{1}_{\{h^{m+1}_{i,j}\leq h^{m+1}_{i,j+1}\}}}{\Delta y^2}\\
        &-\tfrac{(\rho^{m*,o}_{i,j,k} - \rho^{m*,l}_{i,j,k})\mathbf{1}_{\{h^{m+1}_{i-1,j}> h^{m+1}_{i,j}\}}}{\Delta x^2} -\tfrac{(\rho^{m*,o}_{i,j,k} - \rho^{m*,b}_{i,j,k})\mathbf{1}_{\{h^{m+1}_{i,j-1}> h^{m+1}_{i,j}\}}}{\Delta y^2}\\
        &A_{4,i,j,k}^{m*} = \tfrac{\rho^{m*,o}_{i+1,j,k} + \rho^{m*,l}_{i+1,j,k}}{2\Delta x^2} + \tfrac{(\rho^{m*,o}_{i+1,j,k} - \rho^{m*,l}_{i+1,j,k})\mathbf{1}_{\{h^{m+1}_{i,j}> h^{m+1}_{i+1,j}\}}}{\Delta x^2}
        \\
        &A_{5,i,j,k}^{m*} = \tfrac{\rho^{m*,o}_{i,j+1,k} + \rho^{m*,b}_{i,j+1,k}}{2\Delta y^2} + \tfrac{(\rho^{m*,o}_{i,j+1,k} - \rho^{m*,b}_{i,j+1,k})\mathbf{1}_{\{h^{m+1}_{i,j}> h^{m+1}_{i,j+1}\}}}{\Delta y^2}.
        \\
        \end{aligned}
    \end{equation*}
    Here $A_{1,i,j,k}^{m*}\approx \tfrac{\rho_{i-1,j,k}^{m*,o}}{\Delta x^2}$, $A_{2,i,j,k}^{m*}\approx \tfrac{\rho_{i,j-1,k}^{m*,o}}{\Delta y^2}$, $A_{3,i,j,k}^{m*}\approx -\tfrac{2\rho_{i,j,k}^{m*,o}}{\Delta x^2} - \tfrac{2\rho_{i,j,k}^{m*,o}}{\Delta y^2}$, $A_{4,i,j,k}^{m*}$$\approx \tfrac{\rho_{i+1,j,k}^{m*,o}}{\Delta x^2}$, $A_{5,i,j,k}^{m*}\approx \tfrac{\rho_{i,j+1,k}^{m*,o}}{\Delta y^2}$, thus the summation on the right hand side of \eqref{dis_split2} is a discretization of the term $\nabla_{\mathbf{x}}\cdot(D(z)\nabla_{\mathbf{x}} \rho^{z})$.   
    The specific forms of $A_{1,i,j,k}^{m*}$, $A_{2,i,j,k}^{m*}$, $A_{3,i,j,k}^{m*}$, $A_{4,i,j,k}^{m*}$ and $A_{5,i,j,k}^{m*}$ are respectively determined by the coefficients in front of $\varrho_{i-1,j}^{m}$, $\varrho_{i,j-1}^{m}$, $\varrho_{i,j}^{m}$, $\varrho_{i+1,j}^{m}$, $\varrho_{i,j+1}^{m}$ in the limiting discretization of the ADM. We will see this from the proof of Theorem \ref{Thm: ap}.
    
\end{itemize}

Furthermore, the discretization of spatial boundary conditions are as follows:
\begin{equation}\label{dis_xy_boundary}
    \begin{aligned}
    &\rho^{m+1}_{0,j,k} = \rho^{m+1}_{1,j,k},\quad \rho^{m+1}_{N_x,j,k} = \rho^{m+1}_{N_x-1,j,k}, \quad \rho^{m+1}_{i,0,k} = \rho^{m+1}_{i,1,k},\quad \rho^{m+1}_{i,N_y,k} = \rho^{m+1}_{i,N_y-1,k},\text{ $\forall$ $i$, $j$, $k$.}\\
    &h^{m+1}_{0,j} = h^{m+1}_{1,j},\quad h^{m+1}_{N_x,j} = h^{m+1}_{N_x-1,j}, \quad h^{m+1}_{i,0} = h^{m+1}_{i,1},\quad h^{m+1}_{i,N_y} = h^{m+1}_{i,N_y-1},\text{ $\forall$ $i$, $j$.}\\
    &n^{m+1}_{0,j} = n^{m+1}_{1,j},\quad n^{m+1}_{N_x,j} = n^{m+1}_{N_x-1,j},\quad n^{m+1}_{i,0} = n^{m+1}_{i,1},\quad n^{m+1}_{i,N_y} = n^{m+1}_{i,N_y-1},\text{ $\forall$ $i$, $j$.}\\
    \end{aligned}
\end{equation}
In summary, the numerical scheme is given by \eqref{dis_AHL}-\eqref{dis_xy_boundary}. In the subsequent part, we prove the AP property of the scheme. 

\begin{lemma}\label{lem1}
Let the sequence $(\rho^{m*,\xi}_{i,j,k})$ satisfies \eqref{dis_split1}. When $\kappa$ $\to$ $\infty$, we have $\rho^{m*,\xi}_{i,j,k} = \hat\varrho_{i,j}^{m}\delta _{k = Z^{m,\xi}_{i, j}}$, with  $Z^{m,\xi}_{i, j} := \frac{L_{i,j}^{m+1,\xi}}{\Delta z}$ and $\hat\varrho_{i,j}^{m} :=\sum_{k=0}^{N_z}\rho^{m}_{i,j,k}$, where $L_{i,j}^{m+1,\xi} = \mathcal{R}\left(\tfrac{L(h^{m+1,\xi}_{i,j})}{\Delta z}\right)\Delta z$.
\end{lemma}
\begin{proof}
Taking $\kappa \to \infty$ in \eqref{dis_split1}, we obtain that $J^{m*,\xi}_{i,j,k+\frac{1}{2}} = J^{m*,\xi}_{i,j,k-\frac{1}{2}}$. From \eqref{dis_gzh}, we have $g_{i,j,k}^{m+1,\xi} = r\Big(L^{m+1,\xi}_{i,j} - z_k\Big)$ and then for $k \ne Z^{m,\xi}_{i, j}=\frac{L_{i,j}^{m+1,\xi}}{\Delta z}$, the limit satisfies $J^{m*,\xi}_{i,j,k+\frac{1}{2}} = 0$ due to \eqref{dis_gzh}-\eqref{dis_z_boundary}, i.e., $\rho^{m*,\xi}_{i,j,k} = 0$ for $k \ne Z^{m,\xi}_{i, j}$. Moreover, by using the no-flux boundary condition, we deduce that $\hat\varrho_{i,j}^{m*,\xi} :=\sum_{k=0}^{N_z}\rho^{m*,\xi}_{i,j,k}=\sum_{k=0}^{N_z}\rho^{m}_{i,j,k} :=\hat\varrho_{i,j}^{m}$. Thus, $\rho^{m*,\xi}_{i,j,k} = \hat\varrho_{i,j}^{m}\delta _{k =  Z^{m,\xi}_{i j}}$. The proof is completed.
\end{proof}

Now we prove the AP property of \eqref{dis_AHL}-\eqref{dis_xy_boundary} as follows:
\begin{theorem}\label{Thm: ap}
As $\kappa$ $\to$ $\infty$, the sequence $(\rho^{m}_{i,j,k})_{i,j,k,m}$ computed by the scheme \eqref{dis_AHL}-\eqref{dis_xy_boundary} converges to $(\hat\varrho^{m}_{i,j}\delta_{k = Z^{m,\xi}_{i j}})_{i,j,k,m}$ $formally$. Define $\varrho^{m}_{i,j} := \hat\varrho^{m}_{i,j}\Delta z$,  and $\varrho^{m}_{i,j}$ satisfies the following scheme
 \begin{equation}\label{dis_split2_lim}
        \begin{aligned}
        \frac {\varrho^{m+1}_{i,j}-\varrho^{m}_{i,j}}{\Delta t}=&\mathcal{A}_{x,i,j}^m + \mathcal{A}_{y,i,j}^m  + \frac{\hat f^m_{i+\frac12,j} - \hat f^m_{i-\frac12,j}}{\Delta x} + \frac{\hat f^m_{i,j+\frac12} - \hat f^m_{i,j-\frac12}}{\Delta y} + r n_{i,j}^{m+1}\varrho_{i,j}^m,\\
        \end{aligned}
    \end{equation}
with
\begin{equation*}
    \begin{aligned}
    &\mathcal{A}_{x,i,j}^m = \frac{D^{m+1}_{i-\frac12,j}\varrho_{i-1,j}^m - (D^{m+1}_{i-\frac12,j}+D^{m+1}_{i+\frac12,j})\varrho_{i,j}^m + D^{m+1}_{i+\frac12,j}\varrho_{i+1,j}^m}{\Delta x^2},\\
    &\mathcal{A}_{y,i,j}^m=
    \frac{D^{m+1}_{i,j-\frac12}\varrho_{i,j-1}^m - (D^{m+1}_{i,j-\frac12}+D^{m+1}_{i,j+\frac12})\varrho_{i,j}^m + D^{m+1}_{i,j+\frac12}\varrho_{i,j+1}^m}{\Delta y^2},\\
    &\hat f_{i+\frac12,j}^m =\frac{[D^{m+1}_{i+1,j}-D^{m+1}_{i,j}]^+\varrho_{i+1,j}^m -[D^{m+1}_{i+1,j}-D^{m+1}_{i,j}]^-\varrho_{i,j}^m}{\Delta x},\\
    &\hat f_{i,j+\frac12}^m =\frac{[D^{m+1}_{i,j+1}-D^{m+1}_{i,j}]^+\varrho_{i,j+1}^m -[D^{m+1}_{i,j+1}-D^{m+1}_{i,j}]^-\varrho_{i,j}^m}{\Delta y},\\
    \end{aligned}
\end{equation*}
where $D^{m+1}_{i,j}:=D(L_{i,j}^{m+1})$ and $D^{m+1}_{i\pm\frac{1}{2},j}=\tfrac{D^{m+1}_{i\pm1,j}+D^{m+1}_{i,j}}{2}$, $D^{m+1}_{i,j\pm\frac{1}{2}}=\frac{D^{m+1}_{i,j\pm1}+D^{m+1}_{i,j}}{2}$. Furthermore, from \eqref{dis_AHL} and \eqref{dis_nutr}, the limits of $h_{i,j}^{m}$ and $n_{i,j}^{m}$ satisfy
\begin{equation}\label{dis_AHL_limt}
\begin{aligned}
\frac {h^{m*}_{i,j}-h^m_{i,j}}{\Delta t/2}=&D_h[\delta_{x x}(h_{i,j}^{m*}) + \delta_{y y}(h_{i,j}^{m})]+\alpha \varrho^{m}_{i,j} - \beta h^{m*}_{i,j},\\
\frac {h^{m+1}_{i,j}-h^{m*}_{i,j}}{\Delta t/2}=&D_h[\delta_{x x}(h_{i,j}^{m*}) + \delta_{y y}(h_{i,j}^{m+1})] + \alpha \varrho^{m}_{i,j} - \beta h^{m+1}_{i,j},\\
\end{aligned}
\end{equation}
and
\begin{equation}\label{dis_nutr_lim}
\begin{aligned}
\frac {n^{m*}_{i,j}-n^m_{i,j}}{\Delta t/2}=&D_n[\delta_{x x}(n_{i,j}^{m*}) + \delta_{y y}(n_{i,j}^{m})] - \gamma n^{m*}_{i,j}\varrho^{m}_{i,j},\\
\frac {n^{m+1}_{i,j}-n^{m*}_{i,j}}{\Delta t/2}=&D_n[\delta_{x x}(n_{i,j}^{m*}) + \delta_{y y}(n_{i,j}^{m+1})] - \gamma n^{m+1}_{i,j}\varrho^{m}_{i,j},\\
\end{aligned}
\end{equation}
respectively. The scheme \eqref{dis_split2_lim}-\eqref{dis_nutr_lim} is a consistent and stable discretization for ADM.
\end{theorem}

\begin{proof}
Using Lemma \ref{lem1}, we obtain that
$\lim\limits_{\kappa\to \infty}\rho^{m*,\xi}_{i,j,k} = \hat\varrho_{i,j}^{m}\delta _{k =  Z^{m,\xi}_{i j}}$. Thus, 
\[\lim\limits_{\kappa\to \infty}\sum_{k=0}^{N_z}D(z_{k})\rho^{m*,\xi}_{i,j,k}\Delta z = D(z_{Z^{m,\xi}_{ij}})\varrho^{m}_{i,j} = D(L^{m+1,\xi}_{i,j})\varrho^{m}_{i,j},\]
for $\xi\in\{l,\,r,\,o,\,b,\,t\}$. By summing up \eqref{dis_split2} over $k$ and multiplying both sides by $\Delta z$, we can deduce \eqref{dis_split2_lim}. Substituting $\varrho^m_{i,j} = \hat\varrho^{m}_{i,j} \Delta z = \sum_{k=0}^{N_z}\rho^{m}_{i,j,k}\Delta z$ into \eqref{dis_AHL} and \eqref{dis_nutr}, we can obtain the discretizations of $h(\mathbf{x},t)$ in \eqref{dis_AHL_limt} and $n(\mathbf{x},t)$ in \eqref{dis_nutr_lim}. When $\Delta t<C\Delta x^2$, \eqref{dis_split2_lim}-\eqref{dis_nutr_lim} is a stable discretization of the limit model ADM. The proof is completed.
\end{proof}

\begin{remark}
The scheme \eqref{dis_AHL}-\eqref{dis_xy_boundary} is designed to investigate the asymptotic behavior of model \eqref{sys:kinetic} with $k_V = r$ when $\kappa$ $\to$ $\infty$, and we can still use it to solve PBDM with $k_V = r n(\mathbf{x},t)$.
\end{remark}


\section{Numerical results}\label{sec4}
In this section, we investigate the asymptotic behavior and stability of model \eqref{sys:kinetic} numerically. Firstly, we verify numerically that the solution of model \eqref{sys:kinetic} converges to the solution of anisotropic diffusion model \eqref{sys:limit} as $\kappa$ becomes larger. Then we demonstrate the stability/instability of the system \eqref{sys:kinetic} around the non-trivial steady states numerically and provide some interesting phenomena for the pattern formation.

\subsection{Performance of the scheme \eqref{dis_AHL}-\eqref{dis_xy_boundary}}\label{subsec4.1}
We set
\begin{equation}\label{para:step1}
\begin{aligned}
L_x = L_y = 2,\quad Z_w = 1.23,\quad \Delta x = \Delta y = 0.01, \quad\Delta t = 2.5\times10^{-5}.
\end{aligned}
\end{equation}
For the model parameters, we choose the same values as in \cite{Xue2018The}:
\begin{equation}\label{para:1}
\begin{aligned}
&h_0 = 0.25,\,\, r = 0.6931,\,\, D_h = 0.9, \,\, \beta = 2D_h,\,\, \alpha = \beta,\,\, D_n = 2r,\,\, \gamma = 3r.
\end{aligned}
\end{equation}
The diffusion coefficient $D(z)$ is an increasing function in $z$, we take 
\[D(z) = \tfrac{z}{2Z_w} + 0.01,\quad L(h) = Z_w\Big(0.5 - 0.5\tanh\big(30(h-h_0)\big)\Big).\] 
The initial data are
\begin{equation*}
\begin{aligned}
&\rho^{z,0} (\mathbf{x},z) = \Big(h_0 + 0.01(\cos 2\pi x + \cos 2\pi y)\Big)\delta(z - L(h_0)),\\
&h^0(\mathbf{x}) = h_0 + 0.01(\cos 2\pi x + \cos 2\pi y ),\\
&n^0(\mathbf{x}) = 0.
\end{aligned}
\end{equation*}

To verify the model convergence, we simulate PBDM based on the scheme \eqref{dis_AHL}-\eqref{dis_xy_boundary} for $\kappa = 8,\,16,\,32,\,64$, and using a fine mesh ($\Delta z = 0.00375$) in $z$. The numerical results of ADM are given by the scheme \eqref{dis_split2_lim}-\eqref{dis_nutr_lim}. The results are displayed in \Cref{fig1} and we can see that as $\kappa$ increases, the solutions of PBDM get closer to the solution of ADM.
To show the AP property of the scheme in \eqref{dis_AHL}-\eqref{dis_xy_boundary}, numerical results of a larger $\Delta z$ ($\Delta z = 0.03$) for $\kappa = 8,\,16,\,32,\,64$, are given in \Cref{fig2}. Similar convergence can be observed,
which indicates that the scheme can capture the right solution behavior when $\Delta z$ does not resolve $\frac{1}{\kappa}$.

\begin{figure}[h]
\centering
\includegraphics [width=10.5cm]{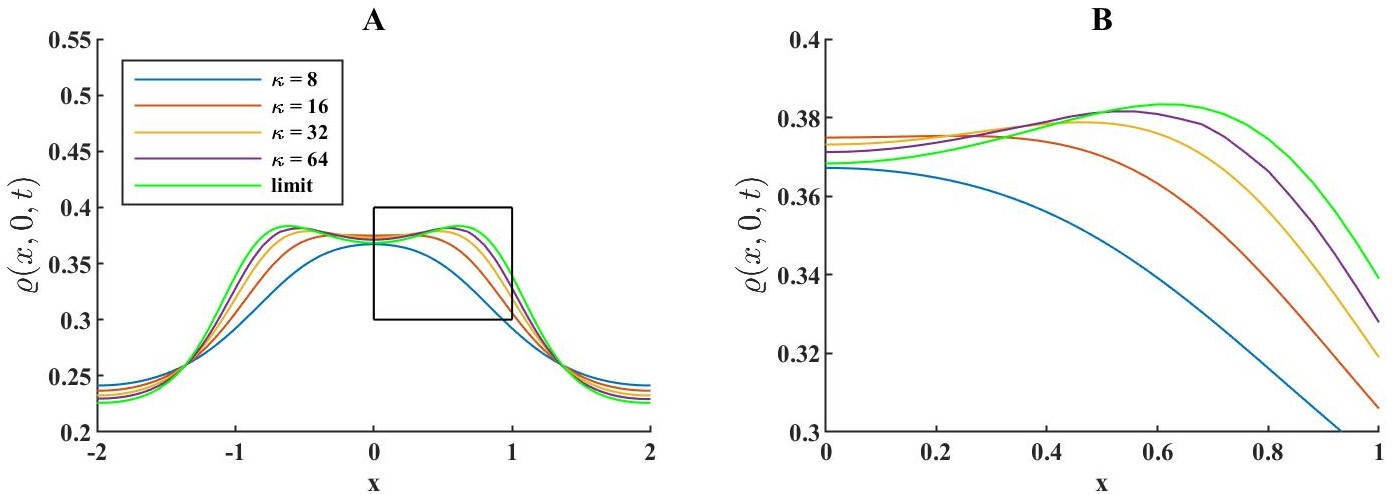}\\
\caption{Verification of the model convergence when $\kappa\to \infty$. A) the numerical results of $\varrho(x,0,t)$ at $t=3$ calculated with $\Delta z = 0.00375$, $\Delta t =2.5\times10^{-5}$ and $\Delta x = \Delta y = 0.01$ for different $\kappa = 8,\,16,\,32,\, 64$. B) the zoom in of the box in figure A.}
\vspace{-0.2cm}
\label{fig1}
\end{figure}

\begin{figure}[h]
\centering
\includegraphics [width=10.5cm]{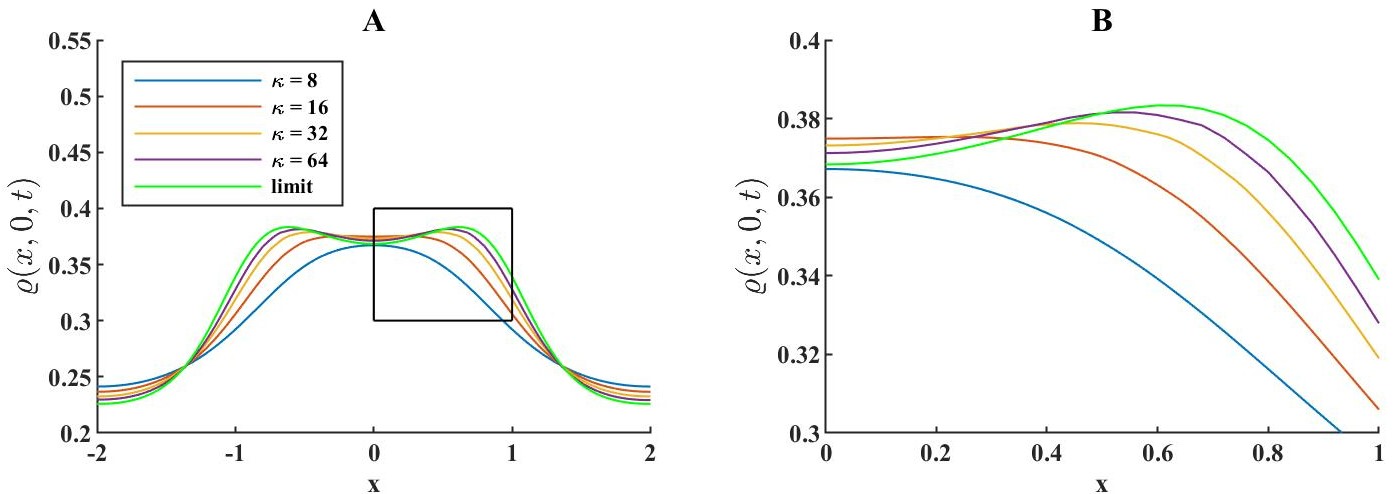}\\
\caption{Verification of the AP property of the scheme. A) the numerical results of $\varrho(x,0,t)$ at $T=3$ calculated with coarse mesh $\Delta z = 0.03$, $\Delta t =2.5\times10^{-5}$ and $\Delta x = \Delta y = 0.01$ for different $\kappa = 8,\,16,\,32,\, 64$. B) the zoom in of the box in figure A.}
\vspace{-0.2cm}
\label{fig2}
\end{figure}

\begin{figure}[h]
\includegraphics [width=13cm]{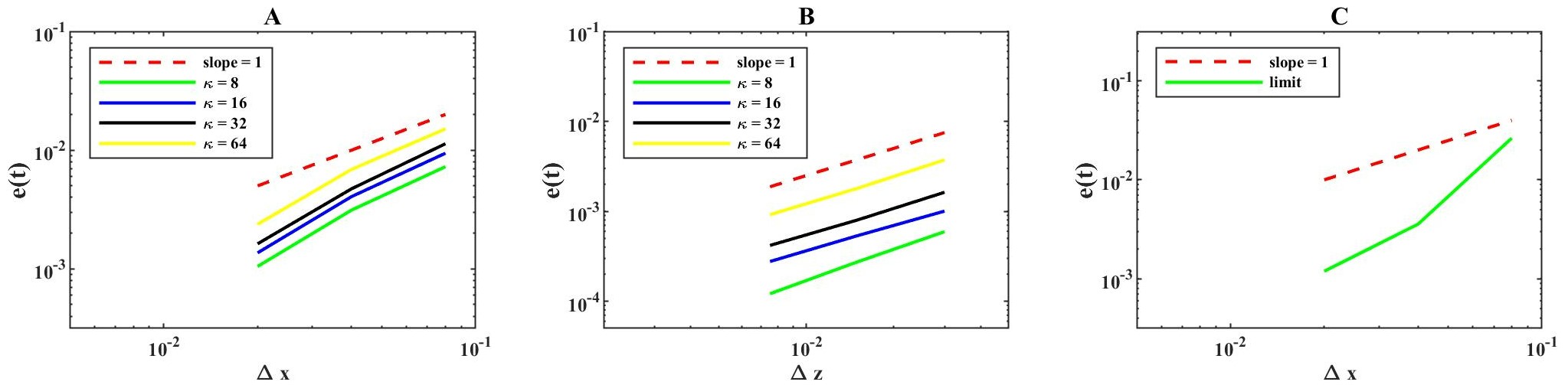}\\
\vspace{-2em}
\caption{Uniform convergence order with respect to $\kappa$. Relative errors defined in \eqref{a20} at $T=3$ for different meshes are displayed. A) the results using different $\Delta x =\Delta y = 0.02$, $0.04$ and $0.08$ with $\Delta t=0.25\Delta x^2$ and $\Delta z=0.00375$; B) the results using different $\Delta z = 0.0075$, $0.015$ and $0.03$ with $\Delta t=0.25\Delta x^2$ and $\Delta x =\Delta y= 0.01$; C) the results using different $\Delta x = 0.02$, $0.04$ and $0.08$ with $\Delta t=0.25\Delta x^2$. The reference solution is calculated with $\Delta t =2.5\times10^{-5}$, $\Delta x = \Delta y = 0.01$ and $\Delta z = 0.00375$. }
\vspace{-0.1cm}
\label{fig3}
\end{figure}

To justify the AP property, we show the uniform convergence order of the scheme for PBDM using different $\kappa= 8,\, 16,\, 32,\, 64$ in Figure \ref{fig3}. The relative errors are defined as follows:
\begin{equation}\label{a20}
    \begin{aligned}
    \textbf{e}(t^m) =\sqrt{\frac{\sum_{i,j}{\sum_{k = 0}^{N_z}\left(\rho^{m}_{i,j,k} - \rho^z(x_i,y_j,z_k,t^m) \right)^2}}{\sum_{i,j}{\sum_{k = 0}^{N_z} \rho^z(x_i,y_j,z_k,t^m)^2}}}.
    \end{aligned}
\end{equation}
Here $\rho^m_{i,j,k}$ is the numerical solution and $\rho^z(x_i,y_j,z_k,t^m)$ is the reference solution obtained numerically by using a very fine mesh such that $\Delta t =2.5\times10^{-5}$, $\Delta x = \Delta y = 0.01$ and $\Delta z = 0.00375$. It is shown in \Cref{fig3} that the scheme \eqref{dis_AHL}-\eqref{dis_xy_boundary} has uniform first order convergence in $\Delta z$, $\Delta x$ and $\Delta y$.

\subsection{Relative deviation function}\label{subsec4.2}
The dependence of $\rho^z$ on the internal state $z$ can induce the instability of PBDM. However, $\varrho$ does not depend on $z$ (since it is the integral of $\rho^z$ over $z$). Thus, even when there is no pattern for $\varrho$, one can not conclude the stability of PBDM. It is necessary to define an auxiliary function to characterize the relative deviation caused by the internal state. We introduce the following relative deviation function
\begin{equation}\label{a23}
\begin{aligned}
R_{\varrho}(t) = \frac{\max_{z\in[0,Z_w]}\{\|\rho^z - \bar\rho^z\|_{\infty}\}}{\max_{z\in[0,Z_w]}\{\|\rho^{z,0} - \bar\rho^z\|_{\infty}\}},\\
\end{aligned}
\end{equation}
where $\|\cdot\|_{\infty}$ denotes the $L^{\infty}$-norm with respect to $\mathbf{x}$ and $\rho^{z,0}$ is the initial data. The function \eqref{a23} characterizes the evolution of the relative deviation of the solution from the steady state $(\bar\rho^z,\bar h,0)$. Note that for the numerical simulation, the relative form \eqref{a23} is well-defined. Theoretically, it may go to infinity, but the numerical values of $\rho^z|_{z = L(\bh)}$ or $\bar\rho^z|_{z = L(\bh)}$ might be large but not infinity.

\subsection{The instability of model \eqref{sys:kinetic} with  $k_V = r$}\label{subsec4.3}

In this subsection, our goal is to characterize the instability of PBDM around the steady states $(\brho^z, \bh, 0)$ and $(0, 0, n_0)$ numerically. Set
\begin{equation}\label{setting}
    \begin{aligned}
    L_x = L_y = 0.5,\,\, Z_w = 1.23,\,\, \Delta x = \Delta y = 0.01,\,\, \Delta z = 0.03,\,\, \Delta t = 2.5\times 10^{-5}.
    \end{aligned}
\end{equation}
Note that the size of the domain is smaller compared with \eqref{para:step1}, which reduces the computational cost. We use the following parameters:
\begin{equation}\label{para:2}
\begin{aligned}
&h_0 = 5,\quad r = 0.6931,\quad D_h = 0.1,\quad \alpha = \beta = 1.8, \quad D_n = 2r,\quad \gamma = 3r,\\
&D(z) = \tfrac{z}{2Z_w} + 0.01,\quad L(h) = Z_w\Big(0.5 - 0.5\tanh\big(30(h-h_0)\big)\Big).
\end{aligned}
\end{equation}
These parameters are not related to the biological experiment. In the numerical simulations, we only consider the case $\bar h \leq h_0$, the case $\bh > h_0$ can be considered similarly, and the details are omitted here.  

\textbf{Case A1. The steady state $(\frac{h_0}{2}\delta(z - L(\tfrac{h_0}{2})),\frac{h_0}{2},0)$. }We choose the initial data to be
\begin{equation}\label{numerical_initial_1}
\begin{aligned}
&\rho^{z,0}(\mathbf{x},z) = \frac{h_0}{2}\delta(z - L(\tfrac{h_0}{2}))  + 0.02\cos{\frac{6\pi x}{L_x}}\cos{\frac{8\pi y}{L_y}}\sin{\frac{2\pi z}{Z_w}},\\
&h^0(\mathbf{x}) = \frac{h_0}{2} + 0.02\cos{\frac{6\pi x}{L_x}}\cos{\frac{8\pi y}{L_y}},\\
&n^0(\mathbf{x}) = 0.\\
\end{aligned}
\end{equation}
It should be noted that the initial condition $\rho^{z,0}(\mathbf{x},z)$ might be negative, which contradicts to the biological background. But our aim is to numerically verify the theoretical results of stability, and this setting is reasonable. \Cref{fig4}A gives the time evolution of $R_{\varrho}(t)$ around the steady state $(\frac{h_0}{2}\delta(z - L(\tfrac{h_0}{2})),\frac{h_0}{2},0)$. We observe that $R_{\varrho}(t)$ increases first and then decreases towards $0$.  This is because $\rho^z$ first goes away from the steady state since all internal state of $\rho^z$ has a tendency to become $z = L(\tfrac{h_0}{2}) \approx Z_w$, which is driven by the advection term $\kappa\partial_z(g(z,h)\rho^z)$. The internal state reaches the steady state $Z_w$ and the diffusion term leads to the uniform spatial distribution of $\rho^z$ for any $z\in[0,Z_w]$ (Note that $\int_{\Lambda}\delta\rho^z\mathrm{d}\mathbf{x}\mathrm{d}z = 0$). Then the system is unstable around the steady state $(\frac{h_0}{2}\delta(z - L(\tfrac{h_0}{2})),\frac{h_0}{2},0)$ even if there is no pattern formation.
\begin{figure}[h]
\centering
\includegraphics [width=10cm]{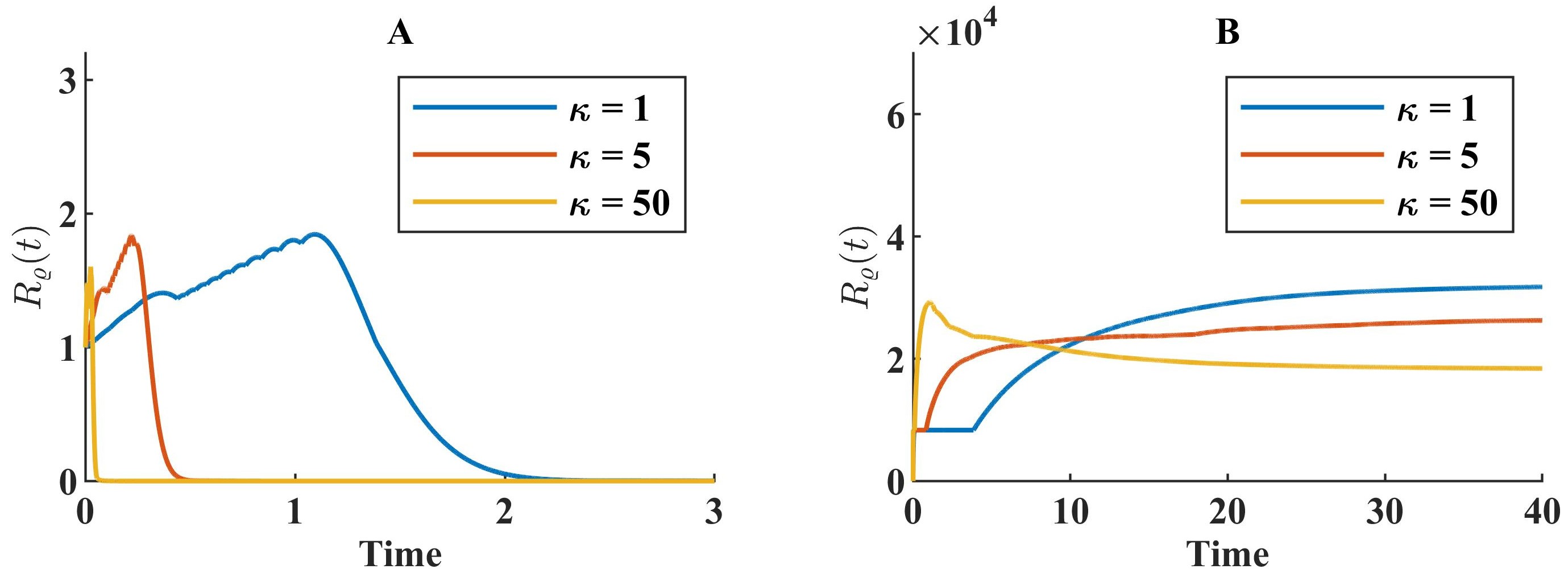}\\
\vspace{-1em}
\caption{Case A1. The evolution of the relative deviation function $R_{\varrho}(t)$ with $k_V = r$ for different $\kappa$. A: The steady state  $(\frac{h_0}{2}\delta(z-L(\tfrac{h_0}{2})), \frac{h_0}{2},0)$ with the initial data \eqref{numerical_initial_1}. B: The steady state $(h_0\delta(z-L(h_0)),h_0,0)$ with the initial data \eqref{numerical_initial_2}. Both steady states are unstable. Note that the larger $\kappa$, the faster $R_{\varrho}(t)$ grows near the origin.} 
\vspace{-0.2cm}
\label{fig4}
\end{figure}

\textbf{Case A1. The steady state $(h_0\delta(z - L(h_0)),h_0,0)$. }The initial data are
\begin{equation}\label{numerical_initial_2}
\begin{aligned}
&\rho^{z,0}(\mathbf{x},z) = h_0\delta(z - L(h_0))  + 0.02\cos{\frac{6\pi x}{L_x}}\cos{\frac{8\pi y}{L_y}}\sin{\frac{2\pi z}{Z_w}},\\
&h ^0(\mathbf{x}) = h_0 + 0.02\cos{\frac{6\pi x}{L_x}}\cos{\frac{8\pi y}{L_y}},\\
&n ^0(\mathbf{x}) = 0.\\
\end{aligned}
\end{equation}
In this case, patterns can form by using $L_x = L_y = 20$, $\Delta t = 0.0025$, $\Delta x = \Delta y = 0.1$ (other parameters remain unchanged). The dynamics of $R_{\varrho}(t)$ and $\varrho(\mathbf{x},t)$ indicate that PBDM is unstable around the steady state $(h_0\delta(z - L(h_0)),h_0,0)$ (see \Cref{fig4}B). Moreover, dot patterns can be observed as in \Cref{fig5}.

\begin{figure}[h]
\centering
\includegraphics [width=13cm]{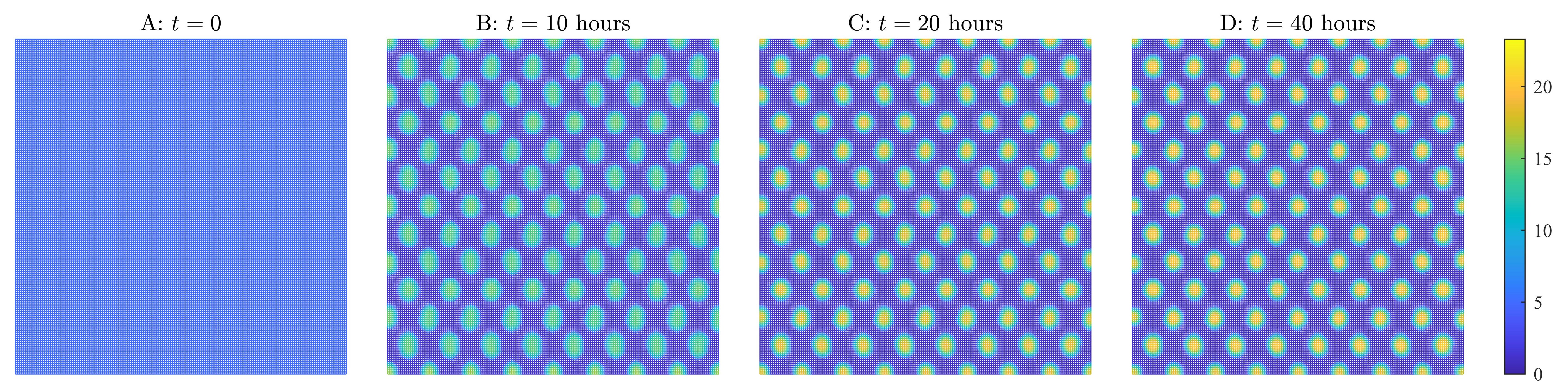}\\
\vspace{-1em}
\caption{Case A1. The time evolution of the density $\varrho(\mathbf{x},t)$ calculated with $k_V = r$, $\kappa = 1$ and initial data \eqref{numerical_initial_2}. It is shown that the steady state is $(h_0\delta(z - L(h_0)),h_0,0)$ is unstable due to the pattern formation of $\varrho(\mathbf{x},t)$. }
\vspace{-0.25cm}
\label{fig5}
\end{figure}

\textbf{Case A2. The steady state $(0,0,0.5)$. }Consider the following initial data
\begin{equation}\label{numerical_initial_3}
    \begin{aligned}
    &\rho^{z,0}(\mathbf{x},z) =  0.02\Big|\cos{\frac{6\pi x}{L_x}}\cos{\frac{8\pi y}{L_y}}\sin{\frac{2\pi z}{Z_w}}\Big|,\\
    &h ^0(\mathbf{x}) = 0.02\Big|\cos{\frac{6\pi x}{L_x}}\cos{\frac{8\pi y}{L_y}}\Big|,\\
    &n ^0(\mathbf{x}) = 0.5 + 0.02\cos{\frac{6\pi x}{L_x}}\cos{\frac{8\pi y}{L_y}}.\\
    \end{aligned}
\end{equation}
\Cref{fig6} shows that the model \eqref{sys:kinetic} is unstable at the steady state $(0,0,0.5)$. Combining with the structure of model, we know that nutrition $n(\mathbf{x},t)$ will be exhausted eventually no matter how small $\rho^{z,0}$ is, i.e., $n(\mathbf{x},t)$ can't go back to the original state $n^0(\mathbf{x},t)>0$. So it is obvious that the steady state $(0,0,\bn)$ is unstable if $\bn$ is positive.

\begin{figure}[http!]
\centering
\vspace{-1em}
\includegraphics [width=10cm]{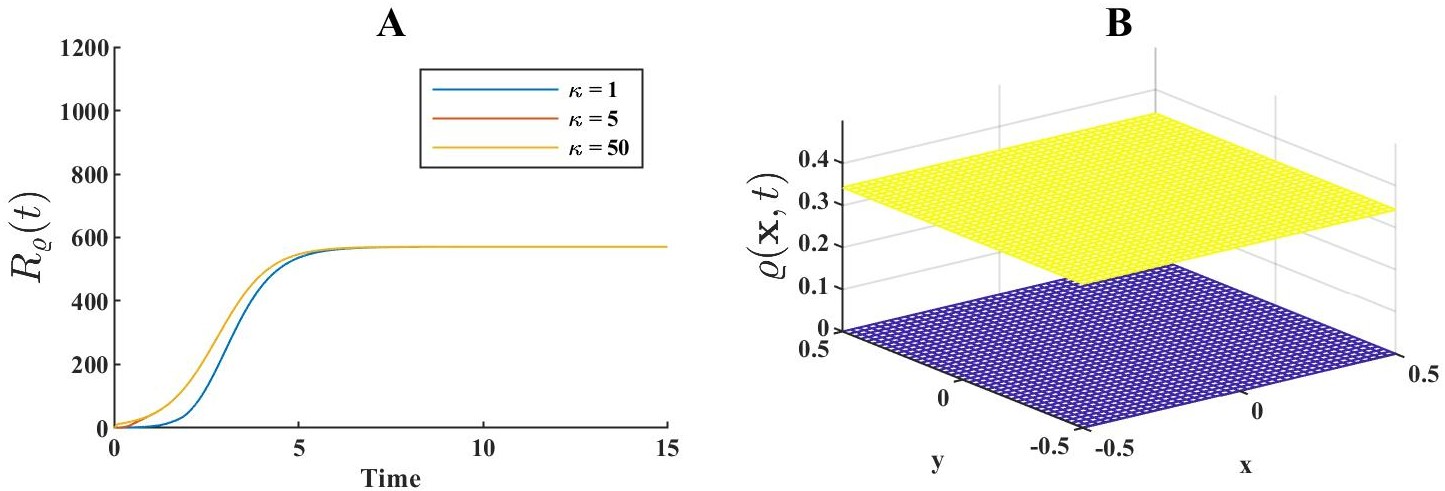}\\
\vspace{-1em}
\caption{Case A2. A) The evolution of the derivation function $R_{\varrho}(t)$ with $k_V = r$ and the initial data \eqref{numerical_initial_3}. B) the state of $\varrho(\mathbf{x},t)$ at $t = 15$h (yellow) and the blue one is $\bar\varrho = 0$. Here we choose different $\kappa$. Note that the larger $\kappa$, the faster $R_{\varrho}(t)$ grows near the origin, but values of $\kappa$ have no influence on $\varrho(\mathbf{x},15)$. It is displayed that the steady state $(0,0,0.5)$ is unstable. }
\label{fig6}
\end{figure}
\subsection{Comparison of the stability results for ADM}
As a comparison, we provide numerical simulations of the limit model (\ref{sys:limit}) at the corresponding steady states $(\bar\varrho, \bar h, \bar n)$. We use the same parameters \eqref{setting} and \eqref{para:2} used in \Cref{subsec4.3}. To characterize the instability/stability of the limit model, we introduce the following relative deviation function
\begin{equation}\label{limit_deviation_funct}
\begin{aligned}
R^*_{\varrho}(t) = \frac{\|\varrho(\mathbf{x}, t) - \bar\varrho\|_{\infty}}{\|\varrho(\mathbf{x}, 0) - \bar\varrho\|_{\infty}}.\\
\end{aligned}
\end{equation}

\Cref{limit_2}A is the evolution of $R_{\varrho}^*(t)$ around the steady state $(\frac{h_0}{2}, \frac{h_0}{2}, 0)$. We observe that $R_{\varrho}^*(t)$ decreases towards $0$ directly, which indicates that the model (\ref{sys:limit}) is stable around the steady state $(\frac{h_0}{2}, \frac{h_0}{2}, 0)$. This is consistent with the curve in \Cref{fig4}A. Compared with (\ref{a23}), $R_{\varrho}^*(t)$ can only characterize the relative deviation between $\varrho(\mathbf{x},t)$ and $\bar\varrho$, but not the internal state due to the $z$-integration. On the other hand, it is displayed in \Cref{limit_2}B that the model (\ref{sys:limit}) is unstable at the steady state $(h_0, h_0, 0)$. \Cref{limit_1} shows that the limit model (\ref{sys:limit}) is unstable at the steady state $(0, 0, 0.5)$. 

\begin{figure}[htbp!]
\centering
\vspace{-1em}
\includegraphics [width=10cm]{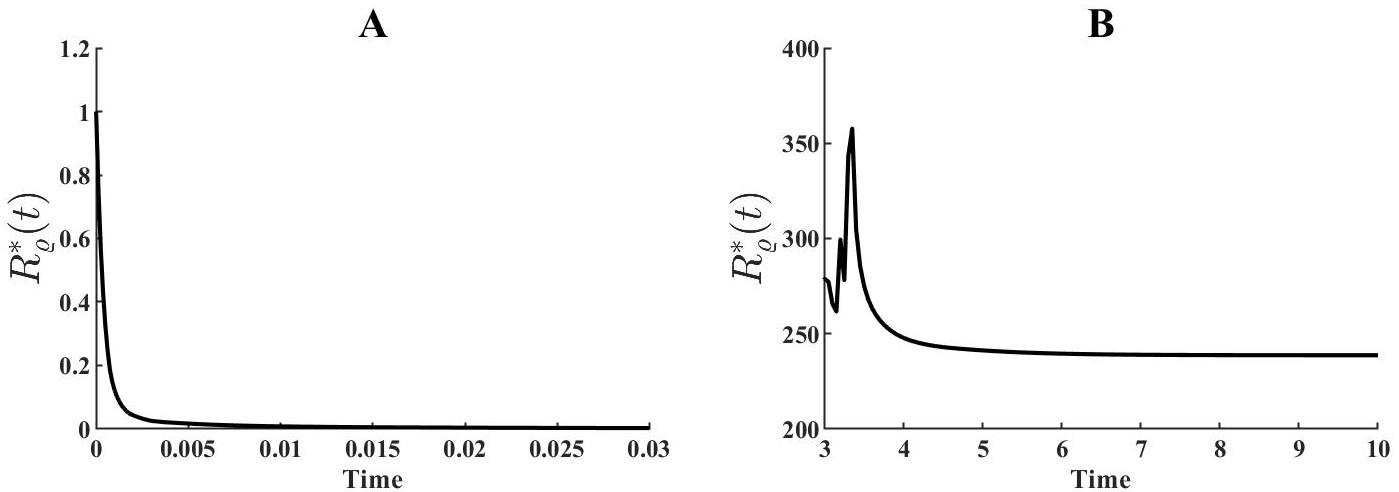}\\
\caption{The stability of ADM around $(\frac{h_0}{2}, \frac{h_0}{2},0)$ and $(h_0,h_0,0)$. A) $(\frac{h_0}{2}, \frac{h_0}{2},0)$. The evolution of the relative deviation function $R^*_{\varrho}(t)$ with the initial data \eqref{numerical_initial_1}.    . B) $(h_0,h_0,0)$.  The evolution of the relative deviation function $R^*_{\varrho}(t)$ with the initial data \eqref{numerical_initial_2}.}
\vspace{-0.2cm}
\label{limit_2}
\end{figure}

\begin{figure}[h]
\centering
\vspace{-1em}
\includegraphics [width=10cm]{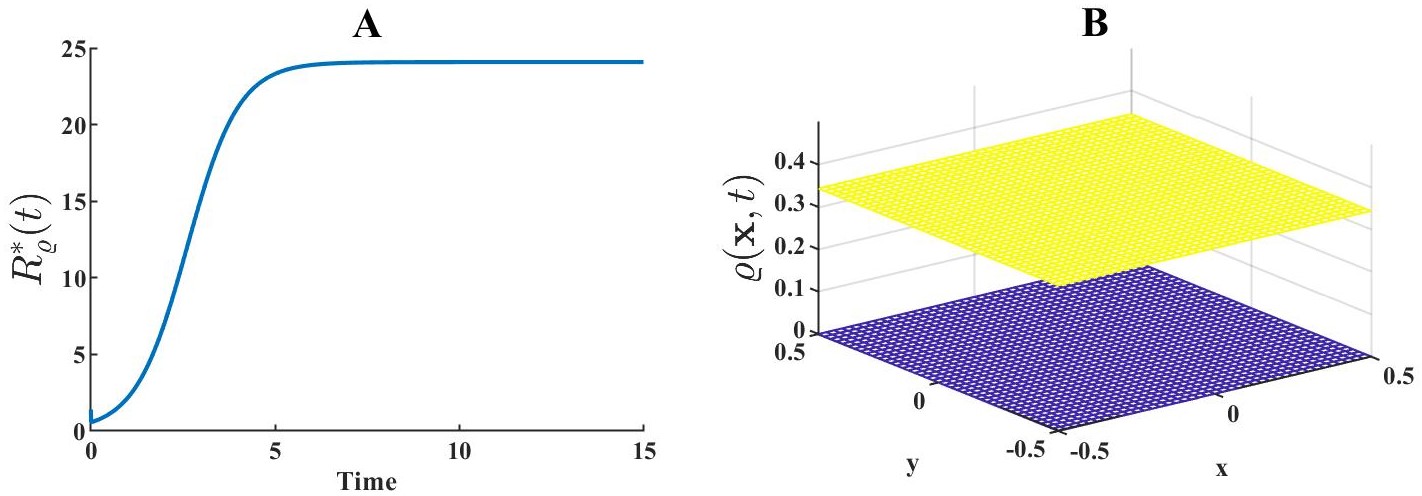}\\
\caption{The stability of ADM around $(0,0,0.5)$. A) the evolution of the derivation function $R^*_{\varrho}(t)$ with the initial data \eqref{numerical_initial_3}. B) The state of $\varrho(\mathbf{x},t)$ at $t = 15$h (yellow) and the blue one is $\bar\varrho = 0$.}
\label{limit_1}
\end{figure}


\subsection{The stability/instability with $k_V = r n(\mathbf{x},t)$}\label{subsec4.4}
In this subsection, we present the numerical simulation of model \eqref{sys:kinetic} with $k_V = r n(\mathbf{x},t)$. We choose the parameters \eqref{setting} and \eqref{para:2} used in \Cref{subsec4.3}. For the steady state $(\brho^z,\bh,0)$, we consider the initial data \eqref{numerical_initial_1} and \eqref{numerical_initial_2}. \Cref{fig7} shows that $R_{\varrho}(t)$ monotonically decreases to $0$, i.e., the model \eqref{sys:kinetic} with $k_V = r n(\mathbf{x},t)$ is stable at the steady states.

As we mentioned in Subection \ref{subsec:Preliminaries}, the distribution of $\bar \rho^z$ with respect to $z$ is uncertain when $k_V = r n(\mathbf{x},t)$. When $\brho^z$ is a continuous function, we consider the following initial data:
\begin{equation}\label{numerical_initial_4}
\begin{aligned}
&\rho^{z,0}(\mathbf{x},z) = \frac{h_0}{2Z_w} + 0.02\cos{\frac{6\pi x}{L_x}}\cos{\frac{8\pi y}{L_y}}\sin{\frac{2\pi z}{Z_w}},\\
\end{aligned}
\end{equation}
and
\begin{equation}\label{numerical_initial_5}
\begin{aligned}
&\rho^{z,0}(\mathbf{x},z) = \frac{h_0}{Z_w}  + 0.02\cos{\frac{6\pi x}{L_x}}\cos{\frac{8\pi y}{L_y}}\sin{\frac{2\pi z}{Z_w}},\\
\end{aligned}
\end{equation}
with $h^0(\mathbf{x})$ and $n^0(\mathbf{x})$ set in \eqref{numerical_initial_1}, \eqref{numerical_initial_2} respectively.\Cref{fig8} displays that $R_{\varrho}(t)$ decreases to $0$ monotonically, then the model \eqref{sys:kinetic} is both stable around the steady state $(\tfrac{h_0}{2Z_w}, \tfrac{h_0}{2},0)$ and $(\tfrac{h_0}{Z_w}, h_0, 0)$.
\begin{figure}[htbp!]
\centering
\vspace{-1em}
\includegraphics [width=10cm]{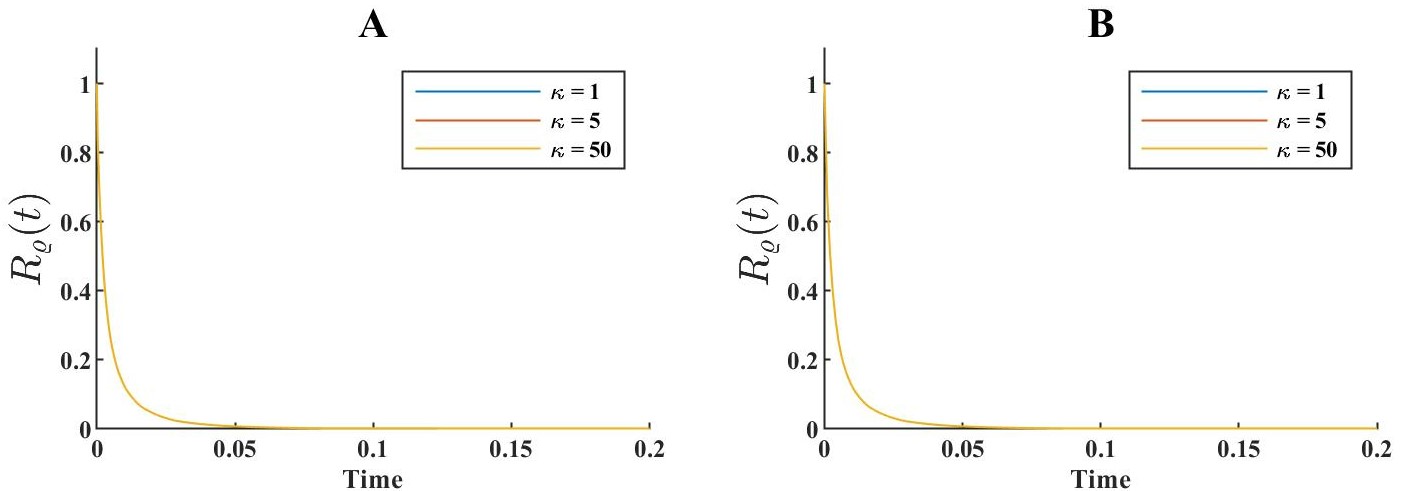}\\
\vspace{-1em}
\caption{Case B1. The evolution of $R_{\varrho}(t)$ with $\brho^z = \bh\delta(z-L(\bar h))$ for different $\kappa$. A: The steady state $(\frac{h_0}{2}\delta(z-L(\tfrac{h_0}{2})),\frac{h_0}{2},0)$ with the initial date \eqref{numerical_initial_1}. B: The steady state $(h_0\delta(z-L(h_0)), h_0, 0)$ with the initial data \eqref{numerical_initial_2}. The values of $\kappa$ have little influence on $R_{\varrho}(t)$. Both steady states are stable. }
\label{fig7}
\end{figure}
\begin{figure}[htbp!]
\centering
\vspace{-1em}
\includegraphics [width=10cm]{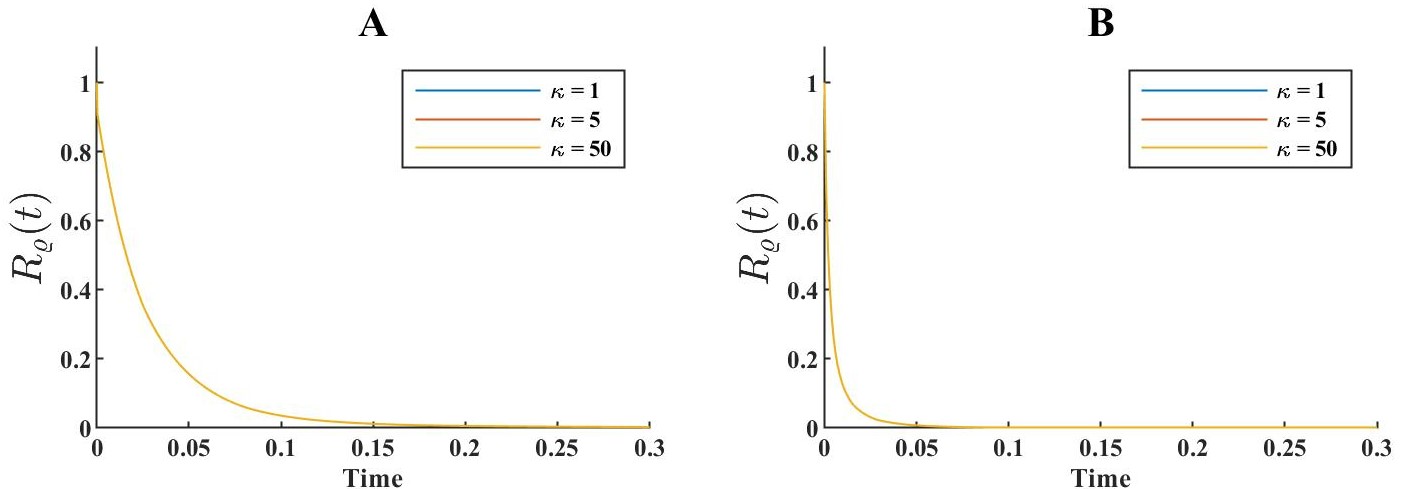}\\
\vspace{-1em}
\caption{Case B1. The evolut.ion of $R_{\varrho}(t)$ when $\brho^z$ is continuous  A) The steady state $(\frac{h_0}{2Z_w},\frac{h_0}{2},0)$ with the initial date \eqref{numerical_initial_4}. B) The steady state $(\frac{h_0}{Z_w},h_0,0)$ with the initial data \eqref{numerical_initial_5}. These two steady states both are stable.}
\label{fig8}
\end{figure}

When considering the state $(0,0,\bar n )$, we still choose the initial data \eqref{numerical_initial_3}. \Cref{fig9} shows that model \eqref{sys:kinetic}  with $k_V = r n(\mathbf{x},t) $ is unstable around the steady state $(0,0,0.5)$.

\begin{figure}[htbp!]
\centering
\vspace{-0.5em}
\includegraphics [width=10cm]{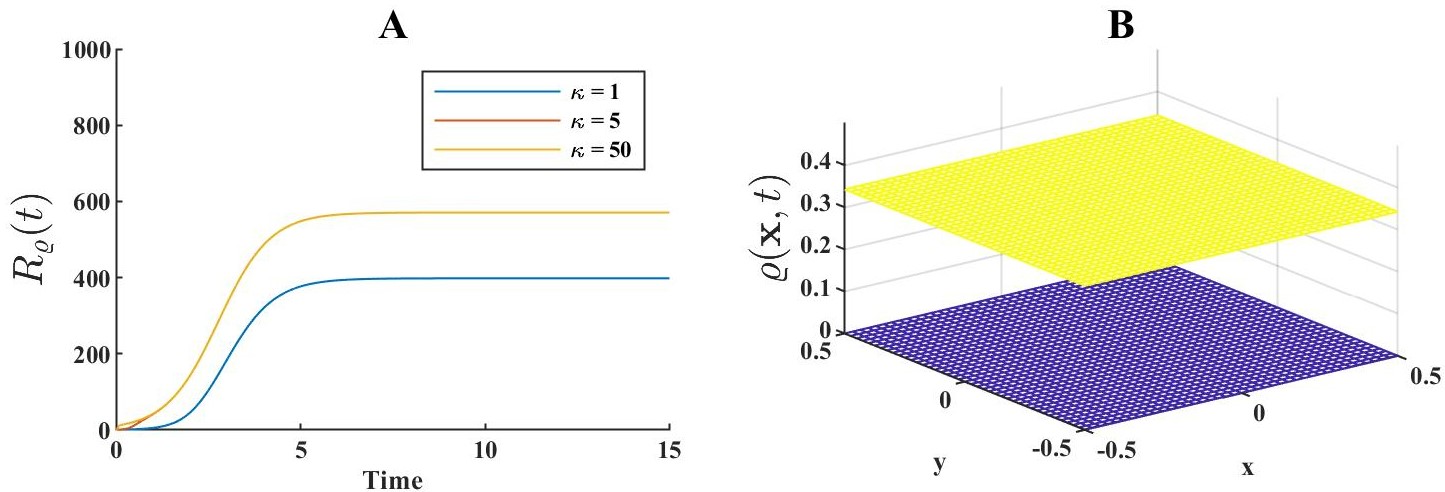}\\
\vspace{-1em}
\caption{Case B2. Instability with $k_V = r n(\mathbf{x},t)$ at $(0,0,0.5)$. A) The evolution of $R_{\varrho}(t)$  in the state $(0,0,\bar n)$ with the initial data \eqref{numerical_initial_3}. B: The state of $\varrho(\mathbf{x},t)$ at $t = 15$ hours (yellow) and the blue one is $\bar\varrho = 0$. Here we choose different $\kappa$. Note that the larger $\kappa$, the faster $R_{\varrho}(t)$ grows near the origin, but values of $\kappa$ have no influence on $\varrho(\mathbf{x},15)$.}
\label{fig9}
\end{figure}

\subsection{Interesting patterns}\label{subsec4.6}
\,\,In this subsection, we present some interesting spatial patterns. 

\textbf{Effect of initial spatial distribution.}
\,\,We show some numerical simulations of both PBDM and ADM with different initial data, which illustrate that the distribution of $\rho^{z,0}$ can influence the pattern formation of $\varrho(\mathbf{x},t)$, such as circular sector pattern. 

We choose
\begin{equation}\label{a33}
    \begin{aligned}
    L_x = L_y = 60,\,\, Z_w = 1.23,\,\, \Delta x = \Delta y = 0.1,\,\, \Delta z = 0.03,\,\, \Delta t = 0.0025,\,\,\kappa = 1.
    \end{aligned}
\end{equation} The parameters in PBDM are the same as in \eqref{para:1} and we take 
\begin{equation}\label{Dz_L(h)}
    \begin{aligned}
    D(z) = \frac{1}{2}\left(\frac{z}{Z_w}\right)^{30} + 0.01,\quad L(h) = Z_w\Big(0.5 - 0.5\tanh\big(1000(h-h_0)\big)\Big).
    \end{aligned}
\end{equation}
Note that $D(z)$ is a monotonically increasing non-negative function of $z$. Unless otherwise specified, we use these step sizes and parameters in the remaining simulations. The following three initial conditions are considered:
\begin{subequations}\label{s51}
\begin{align}
    \text{\textbf{Case 1}:}\,\,\, \rho^{z,0} &= \frac{1}{8\pi}\{\exp{[-\frac{1}{2}(\frac{x^2}{16} + y^2)]} + \exp{[-\frac{1}{2}(x^2 + \frac{y^2}{16})]}\}\mathbf{1}_{\{z = Z_w\}}\mathbf{1}_{\mathbf{D}},
 \label{a35.1}\\
     \text{\textbf{Case 2}:}\,\,\,
    \rho^{z,0} &= \frac{1}{2\pi}\mathbf{1}_{\{z = Z_w\}}\mathbf{1}_{\{(x,y) \in [-5,5] \times [-1,1] \cup  [-1,1] \times [-5,5]\}},
 \label{a35.2}\\
     \text{\textbf{Case 3}:}\,\,\,
    \rho^{z,0} &= \frac{1}{2\pi}\mathbf{1}_{\{z = Z_w\}}\mathbf{1}_{\{(x,y) \in [-10,10] \times [-1,1] \cup  [-1,1] \times [-10,10]\}},
 \label{a35.3}
\end{align}
\end{subequations}
with $h|_{t = 0} = 0$, $n|_{t = 0} = 1$, where $\mathbf{D}:=\{(x,y)|x\in [-10,10],\, y\in [-10,10]\}$. Compared with \eqref{a35.1} and \eqref{a35.2}, the spatial distribution of $\rho^{z,0}$ in \eqref{a35.3} is more non-uniform (see \Cref{fig10}).

\begin{figure}[htbp!]
\centering
\vspace{-0.3cm}
\includegraphics [width=11cm]{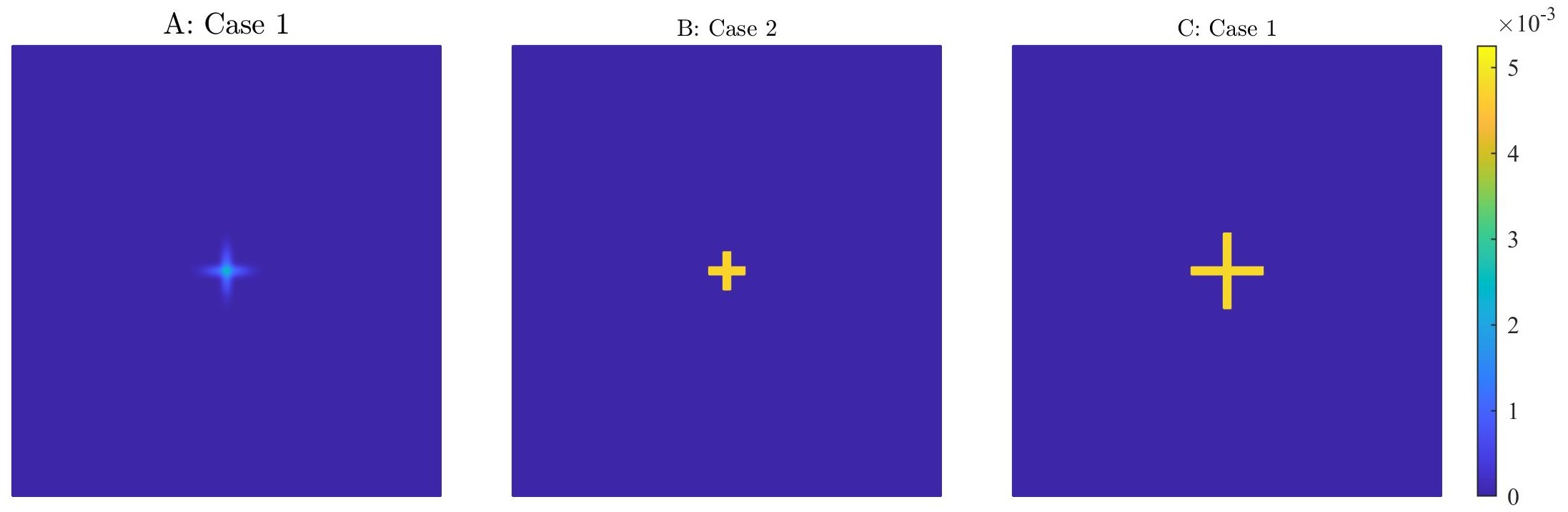}\\
\vspace{-0.1cm}
\caption{The top views of initial data $\varrho(\mathbf{x},0)$ in \eqref{s51}.}
\label{fig10}
\end{figure}

\Cref{fig11} and \Cref{fig11_2} plots respectively the total density $\varrho(\mathbf{x},t)$ in \eqref{sys:kinetic} and \eqref{sys:limit} at $t = 20$h with three different initial data \eqref{a35.1}-\eqref{a35.3}. The circular patterns with alternating high and low cell densities can be observed in all these cases. The total density of cells is low initially, and cells diffuse freely. As time goes on, the density of cells increases, the AHL concentration exceeds the threshold $h_0$, and then internal state $z$ decreases. As more and more cells move into these regions and get trapped, the high $\varrho(\mathbf{x},t)$-density circular patterns develops. 
\begin{figure}[htbp!]
\centering
\vspace{-0.25cm}
\includegraphics [width=11cm]{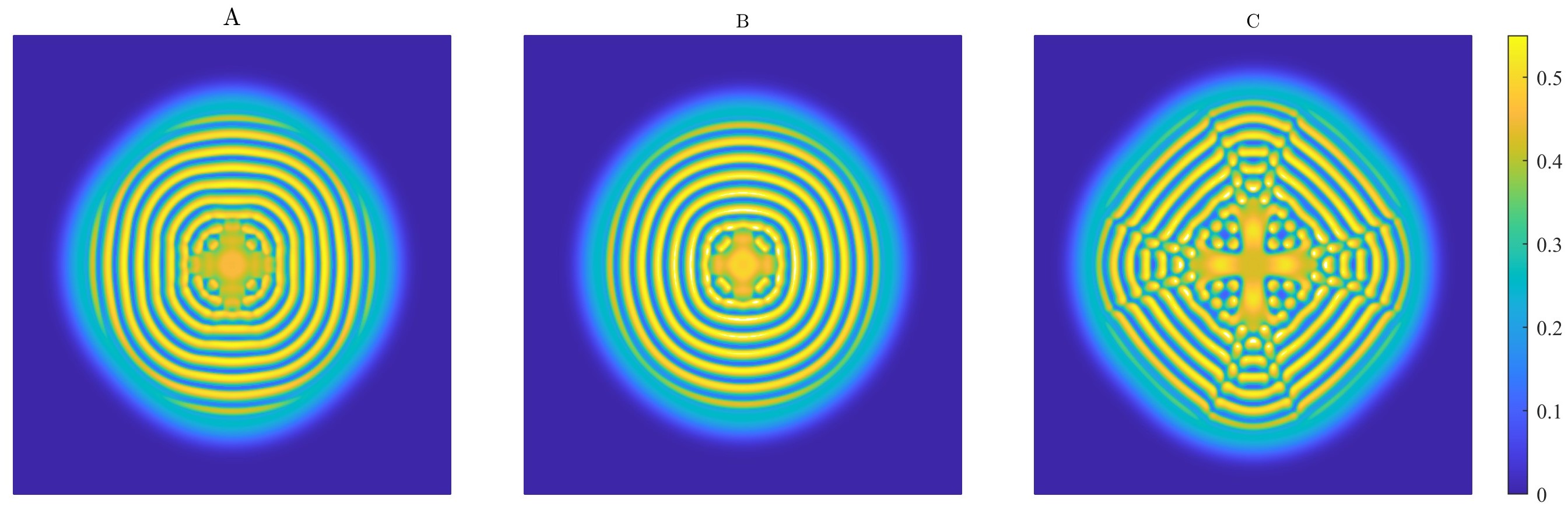}\\
\vspace{-0.1cm}
\caption{The top views of the evolution at $t=20h$ for $\varrho(\mathbf{x},t)$ in \eqref{sys:kinetic} with different initial data and $\kappa = 1$. A: \eqref{a35.1}. B: \eqref{a35.2}. C: \eqref{a35.3}.}
\label{fig11}
\end{figure}

\begin{figure}[htbp!]
\centering
\vspace{-0.25cm}
\includegraphics [width=11cm]{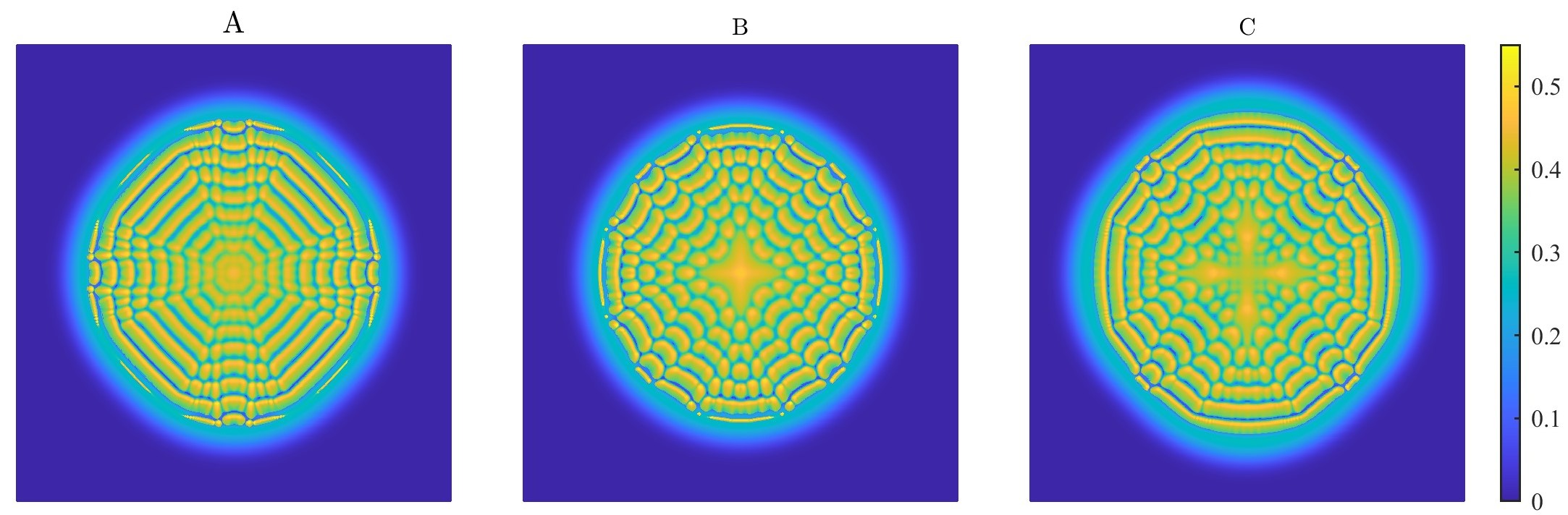}\\
\vspace{-0.1cm}
\caption{The top views of the evolution at $t=20h$ for $\varrho(\mathbf{x},t)$ in \eqref{sys:limit} with different initial data. A: \eqref{a35.1}. B: \eqref{a35.2}. C: \eqref{a35.3}.}
\label{fig11_2}
\end{figure}

\textbf{Interaction of two rings.}\,\,  It is observed that stripes with alternating high and low cell densities establish sequentially behind a radially propagating colony front, one interesting question is: what happens if two stripes interact with each other? It is also a problem of coexistence. In order to show this process, we set
$L_x = 60$, $L_y = 20$,
and other parameters remain unchanged. The initial data is chosen as
\begin{equation}\label{numerical_initial_6}
    \begin{aligned}
    \rho^z(x,y,z,0) = \frac {1}{8\pi}\left\{e^{-0.125\left[x^2+(y - L_y)^2\right]}\mathbf{1}_{\{y\ge 0\}}
    + e^{-0.125\left[x^2+(y + L_y)^2\right]}\mathbf{1}_{\{y < 0\}}\right\}\mathbf{1}_{\{z = Z_w\}},\\
    \end{aligned}
\end{equation}
and $h^0 = 0$, $n^0 = 1$. From the initial data we can see that the cells are near $(0,-L_y)$ and $(0,L_y)$ initially, with the internal steady state $z = Z_w$. The boundary condition with respect to $x$ is the same as \eqref{dis_xy_boundary}. Because the size of $L_y$ is smaller than $L_x$, we use periodic boundary conditions for $y$ such that
\[\rho^z(x,-L_y,z,t) = \rho^z(x,L_y,z,t),\,\, h(x,-L_y,t) = h(x,L_y,t),\,\, n(x,-L_y,t) = n(x,L_y,t),\]
The time evolutions of $\varrho(\mathbf{x},t)$ for both PBDM and ADM are shown in \Cref{fig14} and \Cref{fig14_2}.




\begin{figure}[htbp!]
\centering
\vspace{-0.25cm}
\includegraphics [width=12cm]{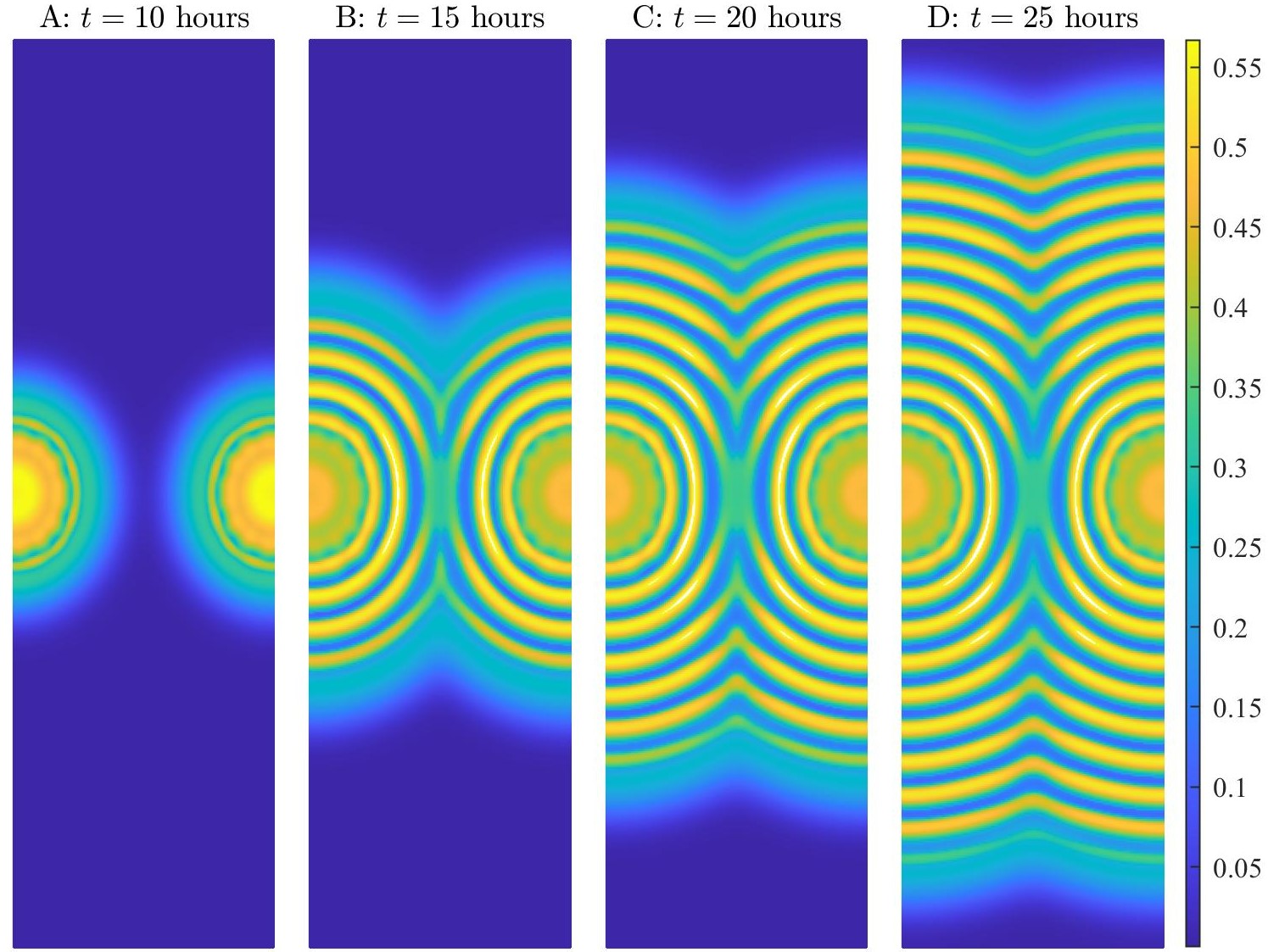}\\
\vspace{-0.1cm}
\caption{The intersection of two populations in \eqref{sys:kinetic} with the initial data \eqref{numerical_initial_6}, the diffusion coefficient $D(z)$ defined in \eqref{Dz_L(h)} and $\kappa = 1$ at different times.}
\label{fig14}
\end{figure}

\begin{figure}[htbp!]
\centering
\vspace{-0.25cm}
\includegraphics [width=12cm]{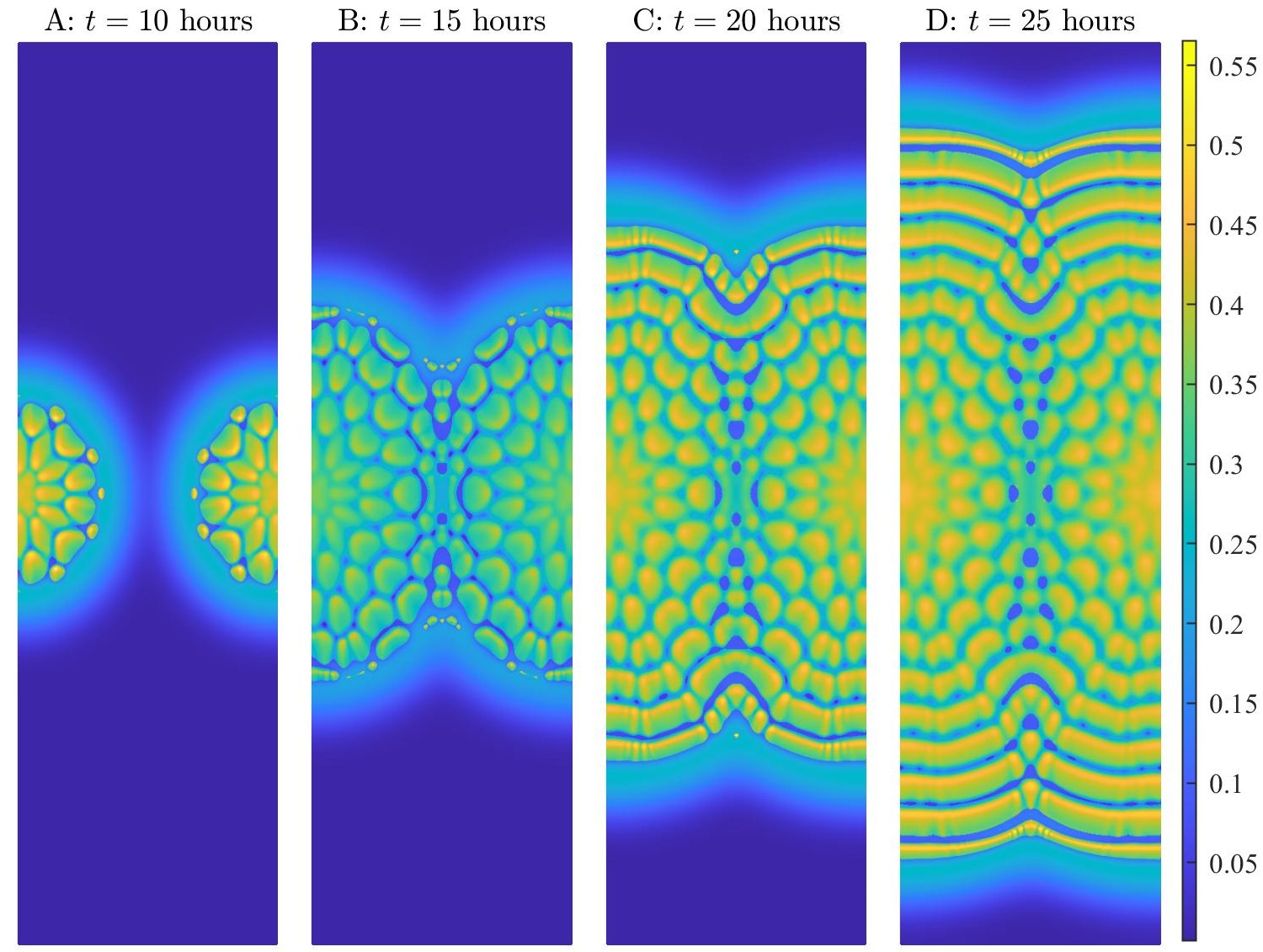}\\
\vspace{-0.1cm}
\caption{The intersection of two populations in \eqref{sys:limit} with the initial data \eqref{numerical_initial_6}, and the diffusion coefficient $D(z)$ defined in \eqref{Dz_L(h)} at different times.}
\label{fig14_2}
\end{figure}

\section{Conclusion}\label{sec5}
The reaction-diffusion system is now widely used to model the motion of bacteria populations, especially in the study of Turing patterns and related issues such as stability and instability. In this work, we investigate the linear stability of a diffusion model with the internal state that characterizes the dynamics of the engineered $Escherichia$ $coli$ populations, which is a new kinetic model derived from the moment closure methods and asymptotic analysis in \cite{Xue2018The}.

At the beginning of this program, we first establish the linear stability analysis of the kinetic model. 
The classical stability analysis fails for our system \eqref{sys:kinetic} due to the delta function distribution. We reformulate the equation according to its structure and then analyze the stability of the system. By introducing small perturbations of the density and concentration in the system, the linearized equations are transformed into a generalized eigenvalue problem.  It is found that the stability of model \eqref{sys:kinetic} depends on the  response speed $\kappa$. Considering $\kappa\to\infty$, we derives an anisotropic diffusion model \eqref{sys:limit} and investigate the similarity and difference between these two models in
terms of linear stability analysis and pattern formation. Then in future work, it could be natural to justify the rigorous stability analysis of the model. At the same time, there are many more interesting topics worthy of further study. 

The second aspect concerns an asymptotic preserving scheme of the \eqref{sys:kinetic}. When the CheZ turnover rate approaches infinity, the pathway-based diffusion model \eqref{sys:kinetic} $formally$ converges to the macroscopic model \eqref{sys:limit} and the Dirac delta form of $\rho^z(\mathbf{x},z,t)$ brings some difficulties to the design of the numerical scheme of model \eqref{sys:kinetic}. We use the time splitting method to solve the $\rho^z$-equation and we design a uniform numerical scheme to capture the Dirac delta that is consistent with a stable discretization of ADM. Due to the difficulties in analysis, we did not justify the limit with $\kappa\to\infty$ from the model \eqref{sys:kinetic} to the limit model \eqref{sys:limit} and left the harder rigorous analysis to further study. Besides, many researchers discussed the hydrodynamic limits and the trend toward equilibrium, see \cite{Gosse2015A, Vauchelet2021Numerical} for instance.

The last important issue is with the numerical simulation viewpoint, which is very useful in studying the patterns, stability, and instability problems. To better understand the stability of system around the steady states, numerical simulations confirm the above theoretical results based on an asymptotic preserving scheme and illustrate that the pattern formation of $\varrho(\mathbf{x},t)$ is just a sufficient condition for the instability of model \eqref{sys:kinetic}. We also show some interesting patterns of cell density $\varrho(\mathbf{x},t)$ to illustrate the fact that the spatial distribution of $\rho^{z,0}$ might lead to the formation of patterns, such as circular sector patterns, the behavior of two populations in contact with each other. Moreover, multiscale models and synthetic biology \cite{Khalil2010Synthetic,Mukherji2009Synthetic, Wang2017A, Xue2009Multiscale} raise more research topics in biological models and mathematical analysis.

\section*{Appendix}
\section{Linear stability analysis of the limit model (\ref{sys:limit}) }\label{SM1}We first determine the non-trivial steady state $(\bar\varrho, \bar h, \bar n)$. Plugging the constant states $(\bar\varrho, \bar h, \bar n)$ into model \eqref{sys:limit}, we have
\begin{equation}\label{limit_constraint}
    \begin{aligned}
    &\Delta _{\mathbf{x}}\bar\varrho = 0,\quad \Delta _{\mathbf{x}}\bar h = 0,\Delta _{\mathbf{x}}\bar n = 0,\\
    &r\bar n\bar\varrho = 0,\quad \alpha \bar\varrho - \beta \bar h = 0,\quad \gamma \bar\varrho\bar n = 0.
    \end{aligned}
\end{equation}
Then we linearize the model \eqref{sys:limit} around the steady state $(\bar\varrho, \bar h, \bar n)$, which satisfies the constraint \eqref{limit_constraint}. Define
\begin{equation}\label{limit_perturbation}
    \varrho = \bar\varrho + \delta \varrho,\quad h = \bar h + \delta h,\quad n = \bar n + \delta n,
\end{equation}
where the perturbation $|\delta \bm{\phi} | \ll 1$, and $\bm{\phi} = (\varrho, h,n)^T $. We can derive the corresponding characteristic matrix and determine the stability according to the sign of real parts of eigenvalues. We assume the perturbation $(\delta\varrho, \delta h, \delta n)$ can be represented by a harmonic wave as follows:
\begin{equation*}
\begin{aligned}
    \left(
  \begin{array}{l}
  \delta \varrho\\
  \delta h\\
  \delta n
  \end{array}
  \right)= \sum_{\mathbf{k}}
      \left(
  \begin{array}{l}
           C_1^{\mathbf{k}}\\
           C_2^{\mathbf{k}}\\
           C_3^{\mathbf{k}}
  \end{array}
  \right)e^{\lambda t + i k_1 x + i k_2 y},
\end{aligned}
\end{equation*}
where $\mathbf{k} = (k_1,k_2)^T \in \mathbb{R}^2$ and $k_1$, $k_2$ are the corresponding frequencies. We also assume that the initial condition is given by
  \begin{align*}
     (\rho^{z,0}(x,z), h^0(x), n^0(x)) = (\brho^{z,0} + \delta \rho^{z,0},\, \bh^0 + \delta h,\, \bn^0 + \delta n)\,.
  \end{align*}
 
\subsection{Linearization}\label{SM1.1} 
Substituting the constraint \eqref{limit_constraint} and the perturbation form \eqref{limit_perturbation} into model \eqref{sys:limit}, we obtain the following perturbation system:
\begin{equation}\label{limit_sys:perturbation}
    \begin{aligned}
    \begin{cases}
        \partial_t \delta\varrho &= \Delta_{\mathbf{x}}\left(D(L(\bar h))\delta\varrho\right) + \Delta_{\mathbf{x}}\left(D_h(L(\bar h))\bar\varrho\delta h\right) + r\bar n \delta\varrho + r\bar\varrho\delta n,\\
        \partial_t \delta h &= D_h\Delta_{\mathbf{x}}\delta h + \alpha\delta\varrho - \beta\delta h,\\
        \partial_t \delta n &= D_n\Delta_{\mathbf{x}}\delta n - \gamma\bar\varrho\delta n - \gamma\bar n\delta \varrho,
    \end{cases}
    \end{aligned}
\end{equation}
where $D_h(\bar h) := \partial_h D(L(\bar h))$.

\subsection{The stability analysis}\label{SM1.2} Substituting the perturbation into the perturbation system \eqref{limit_sys:perturbation} yields the characteristic matrix as follows:
\begin{equation*}
\begin{aligned}
    \left(
  \begin{array}{ccc}
          -KD(L(\bar h)) + r\bar n - \lambda & -KD_h(L(\bar h))\bar\varrho &r\bar\varrho  \\
          \alpha& -KD_h - \beta - \lambda &0\\
          -\gamma \bar n& 0 &-KD_n - \gamma\bar\varrho - \lambda\\
  \end{array}
  \right),
\end{aligned}
\end{equation*}
where $K = k_1^2 + k_2^2$. We focus on the two types of non-zero steady states $(0, 0, \bar n)$ and $(\bar\varrho, \bar h, 0)$. We split the state $\bar h$ into three parts: $\bar h = 0$, $\bar h \ne h_0$ ($\bar h > 0$), $\bar h = h_0$.
\begin{itemize}
    \item \textbf{The steady state $(0,0,\bar n)$ with $\bar n > 0$}. The corresponding eigenvalues are
    \begin{equation*}
        \begin{aligned}
            \lambda_5^1 = -K D(L(0)) + r\bar n,\quad \lambda_5^2 = -K D_h - \beta,\quad \lambda_5^3 = -K D_n\,.
        \end{aligned}
    \end{equation*}
    Observe that the real part of $\lambda_5^1$ are positive if $K$ is sufficiently small. Then the model \eqref{sys:limit} is unstable at the steady state $(0,0,\bar n)$.
    \bigskip
    \item \textbf{The steady state $(\bar\varrho,\bar h, 0)$ with $\bar h \ne h_0$ }. It is noted that as $\bar h$ dose not equal $h_0$, the derivative $D_h(L(\bar h))$ equals $0$ if $\mu$ is sufficiently large. The corresponding eigenvalues are
    \begin{equation*}
        \begin{aligned}
            \lambda_6^1 = -K D(L(\bar h)),\quad \lambda_6^2 = -K D_h - \beta,\quad \lambda_6^3 = -K D_n - \gamma\bar\varrho\,.
        \end{aligned}
    \end{equation*}
    Note that $\lambda_6^i\, (i = 1, 2, 3)$ are negative real values, which means that the model \eqref{sys:limit} is stable at the steady state $(\bar \varrho,\bar h,0)$ as $\bar h$ is away from $h_0$.
    \bigskip
    \item \textbf{The steady state $(\bar\varrho,h_0, 0)$ with $\bar \varrho = \frac{\alpha}{\beta}h_0$} . The characteristic equation is
    \begin{equation*}
       \left[(\lambda + KD(L(h_0)))(\lambda + KD_h + \beta) + \alpha KD_h(L(h_0))\bar\varrho\right](\lambda + KD_n + \gamma\bar\varrho) = 0.
    \end{equation*}
    We can obtain $\lambda_7^3 = -KD_n -\gamma \bar\varrho$ and the eigenvalues $\lambda_7^1$, $\lambda_7^2$ satisfy
        \begin{equation}\label{limit_characteristic_eqn}
      \lambda^2 + (a+b)\lambda + ab + c = 0,
    \end{equation}
    with
    \[a = KD(L(h_0)),\quad b = KD_h + \beta,\quad c = \alpha KD_h(L(h_0))\bar\varrho.\]
    From the definition of $D(L(h))$, we have
    \[\lim_{\mu\to+\infty}(ab + c) = \lim_{\mu\to+\infty} c = \lim_{\mu\to+\infty}D_{h}(L(h_0)) =  -\infty,\]
    for some fixed $K$. By Vieta theorem, there exist a positive root and a negative root for the equation \eqref{limit_characteristic_eqn} for any fixed $K$. Thus the model \eqref{sys:limit} is unstable at the steady state $(\bar\varrho,\bar h, 0)$.
\end{itemize}

\bibliographystyle{alpha}
\bibliography{sample}

\newcommand{\etalchar}[1]{$^{#1}$}
\begin{thebibliography}{VHZBG17}

\bibitem[BGC{\etalchar{+}}05]{Basu2005A}
Subhayu Basu, Yoram Gerchman, Cynthia~H. Collins, Frances~H. Arnold, and Ron
  Weiss.
\newblock A synthetic multicellular system for programmed pattern formation.
\newblock {\em Nature}, 434:1130--1134, 2005.

\bibitem[Bog18]{B18}
Vladimir~I. Bogachev.
\newblock {\em Stationary {F}okker-{P}lanck-{K}olmogorov equations}, volume 229
  of {\em Springer Proc. Math. Stat.}
\newblock Springer, Cham, 2018.

\bibitem[BOS{\etalchar{+}}22]{BOS22}
N.~Bellomo, N.~Outada, J.~Soler, Y.~Tao, and M.~Winkler.
\newblock Chemotaxis and cross-diffusion models in complex environments: models
  and analytic problems toward a multiscale vision.
\newblock {\em Math. Models Methods Appl. Sci.}, 32(4):713--792, 2022.

\bibitem[BSM09]{Baker2009Waves}
Ruth~E. Baker, Santiago Schnell, and Philip~K. Maini.
\newblock Waves and patterning in developmental biology: vertebrate
  segmentation and feather bud formation as case studies.
\newblock {\em The International journal of developmental biology},
  53:783--794, 2009.

\bibitem[CCL08]{cantrell2008approximating}
Robert~Stephen Cantrell, Chris Cosner, and Yuan Lou.
\newblock Approximating the ideal free distribution via
  reaction–diffusion–advection equations.
\newblock {\em J. Differential Equations}, 245:3687--3703, 2008.

\bibitem[CSHW21]{CSH21}
Jianzhi Cao, Hongyan Sun, Pengmiao Hao, and Peiguang Wang.
\newblock Bifurcation and {T}uring instability for a predator-prey model with
  nonlinear reaction cross-diffusion.
\newblock {\em Appl. Math. Model.}, 89:1663--1677, 2021.

\bibitem[DDD19]{DDD19}
Esther~S. Daus, Laurent Desvillettes, and Helge Dietert.
\newblock About the entropic structure of detailed balanced multi-species
  cross-diffusion equations.
\newblock {\em J. Differential Equations}, 266(7):3861--3882, 2019.

\bibitem[EO04]{Erban2004From}
Radek Erban and Hans~G. Othmer.
\newblock From individual to collective behavior in bacterial chemotaxis.
\newblock {\em SIAM J. Appl. Math.}, 65(2):361--391, 2004.

\bibitem[FG09]{Friedl2009Collective}
Peter Friedl and Darren Gilmour.
\newblock Collective cell migration in morphogenesis, regeneration and cancer.
\newblock {\em Nature Reviews Molecular Cell Biology}, 10:445--457, 2009.

\bibitem[FTL{\etalchar{+}}12]{FTL12}
Xiongfei Fu, Lei-Han Tang, Chenli Liu, Jian-Dong Huang, Terence Hwa, and Peter
  Lenz.
\newblock Stripe formation in bacterial systems with density-suppressed
  motility.
\newblock {\em Phys. Rev. Lett.}, 108:198102, May 2012.

\bibitem[Gos15]{Gosse2015A}
Laurent Gosse.
\newblock A well-balanced scheme able to cope with hydrodynamic limits for
  linear kinetic models.
\newblock {\em Appl. Math. Lett.}, 42:15--21, 2015.

\bibitem[Hel92]{Held1992Models}
Lewis~I Held.
\newblock Models for embryonic periodicity.
\newblock {\em Monographs in developmental biology}, 24:1--119, 1992.

\bibitem[HJL17]{Hu2017AsymptoticPreservingSF}
Jingwei Hu, Shi Jin, and Qin Li.
\newblock Asymptotic-preserving schemes for multiscale hyperbolic and kinetic
  equations.
\newblock {\em Handb. Numer. Anal.}, 18:103--129, 2017.

\bibitem[IMN06]{M06}
Masato Iida, Masayasu Mimura, and Hirokazu Ninomiya.
\newblock Diffusion, cross-diffusion and competitive interaction.
\newblock {\em J. Math. Biol.}, 53(4):617--641, 2006.

\bibitem[Jin10]{Jin2010ASYMPTOTICP}
Shi Jin.
\newblock Asymptotic preserving (ap) schemes for multiscale kinetic and
  hyperbolic equations: a review.
\newblock {\em Rivista di Matematica della Universita di Parma}, 2(2):177--216,
  2010.

\bibitem[JLL{\etalchar{+}}22]{jiang2021kinetic}
Ning Jiang, Jiangyan Liang, Yi-Long Luo, Min Tang, and Yaming Zhang.
\newblock On kinetic and macroscopic models for the stripe formation in
  engineered bacterial populations.
\newblock {\em J. Differential Equations}, 323:38--85, 2022.

\bibitem[JSW20]{JSW20}
Hai-Yang Jin, Shijie Shi, and Zhi-An Wang.
\newblock Boundedness and asymptotics of a reaction-diffusion system with
  density-dependent motility.
\newblock {\em J. Differential Equations}, 269(9):6758--6793, 2020.

\bibitem[KC10]{Khalil2010Synthetic}
Ahmad~S. Khalil and James~J. Collins.
\newblock Synthetic biology: applications come of age.
\newblock {\em Nature Reviews Genetics}, 11:367--379, 2010.

\bibitem[KS70]{Keller1970Initiation}
Evelyn~Fox Keller and Lee~A. Segel.
\newblock Initiation of slime mold aggregation viewed as an instability.
\newblock {\em J. Theoret. Biol.}, 26(3):399--415, 1970.

\bibitem[LFL{\etalchar{+}}11]{Liu2011Sequential}
Chenli Liu, Xiongfei Fu, Lizhong Liu, Xiaojing Ren, Carlos~Kwan long Chau,
  Sihong Li, Lu~Xiang, Hualing Zeng, Guanhua Chen, Lei-Han Tang, Peter Lenz,
  Xiaodong Cui, Wei Huang, Terence Hwa, and Jian-Dong Huang.
\newblock Sequential establishment of stripe patterns in an expanding cell
  population.
\newblock {\em Science}, 334(6053):238--241, 2011.

\bibitem[MG21]{Ma2021Bifurcation}
Li~Ma and Shangjiang Guo.
\newblock Bifurcation and stability of a two-species
  reaction–diffusion–advection competition model.
\newblock {\em Nonlinear Analysis-real World Applications}, 59:103241, 2021.

\bibitem[MHH{\etalchar{+}}10]{Marrocco2010Models}
A.~Marrocco, H.~Henry, I.~Barry Holland, Mathis Plapp, Simone~J. S\'{e}ror, and
  Beno{\^i}t Perthame.
\newblock Models of self-organizing bacterial communities and comparisons with
  experimental observations.
\newblock {\em Math. Model. Nat. Phenom.}, 5(1):148--162, 2010.

\bibitem[MPW20]{MPW20}
Manjun Ma, Rui Peng, and Zhian Wang.
\newblock Stationary and non-stationary patterns of the density-suppressed
  motility model.
\newblock {\em Phys. D}, 402:132259, 13, 2020.

\bibitem[Mur02]{Murray2002Mathematical}
James~D. Murray.
\newblock {\em Mathematical Biology}, volume~2.
\newblock Springer, 2002.

\bibitem[MvO09]{Mukherji2009Synthetic}
Shankar Mukherji and Alexander van Oudenaarden.
\newblock Synthetic biology: understanding biological design from synthetic
  circuits.
\newblock {\em Nature Reviews Genetics}, 10:859--871, 2009.

\bibitem[OLF{\etalchar{+}}99]{O'Toole1999The}
Ronan~F. O’Toole, Susanne Lundberg, Sten Fredriksson, Anita Jansson,
  Bo~Nilsson, and Hans Wolf‐Watz.
\newblock The chemotactic response of vibrio anguillarum to fish intestinal
  mucus is mediated by a combination of multiple mucus components.
\newblock {\em Journal of Bacteriology}, 181(14):4308--4317, 1999.

\bibitem[Pat53]{Patlak1953Random}
Clifford~S. Patlak.
\newblock Random walk with persistence and external bias.
\newblock {\em Bulletin of Mathematical Biophysics}, 15:311--338, 1953.

\bibitem[PGK01]{Pittman2001Chemotaxis}
Marc~S. Pittman, Matthew~L. Goodwin, and David~J. Kelly.
\newblock Chemotaxis in the human gastric pathogen helicobacter pylori:
  different roles for chew and the three chev paralogues, and evidence for
  chev2 phosphorylation.
\newblock {\em Microbiology}, 147(9):2493--2504, 2001.

\bibitem[PMO99]{Painter1999Stripe}
Kevin~J. Painter, Philip~K. Maini, and Hans~G. Othmer.
\newblock Stripe formation in juvenile pomacanthus explained by a generalized
  turing mechanism with chemotaxis.
\newblock {\em Proc. Natl. Acad. Sci. USA}, 96(10):5549--5554, 1999.

\bibitem[PSTY20]{Perthame2020Multiple}
Benoit Perthame, Weiran Sun, Min Tang, and Shugo Yasuda.
\newblock Multiple asymptotics of kinetic equations with internal states.
\newblock {\em Math. Models Methods Appl. Sci.}, 30(06):1041--1073, 2020.

\bibitem[SPJ06]{Singh2006Biofilms}
Rajbir Singh, Debarati Paul, and Rakesh~Kumar Jain.
\newblock Biofilms: implications in bioremediation.
\newblock {\em Trends in Microbiology}, 14(9):389--397, 2006.

\bibitem[ST17]{Sun2017Macroscopic}
Weiran Sun and Min Tang.
\newblock Macroscopic limits of pathway-based kinetic models for e. coli
  chemotaxis in large gradient environments.
\newblock {\em SIAM Journal on Multiscale Modeling and Simulation},
  15(2):797--826, 2017.

\bibitem[STY14]{Si2014A}
Guangwei Si, Min Tang, and Xin Yang.
\newblock A pathway-based mean-field model for e. coli chemotaxis: Mathematical
  derivation and keller-segel limit.
\newblock {\em SIAM Journal on Multiscale Modeling and Simulation},
  12(2):907--926, 2014.

\bibitem[SWOT12]{Si2012Pathway}
Guangwei Si, Tailin Wu, Qi~Ouyang, and Yuhai Tu.
\newblock Pathway-based mean-field model for escherichia coli chemotaxis.
\newblock {\em Phys. Rev. Lett.}, 109(4):048101, 2012.

\bibitem[Tan21]{Min2021MODELING}
Min Tang.
\newblock Modeling, analysis and computational methods in population level
  chemotaxis.
\newblock {\em J. Numer. Methods Comput. Appl.}, 42(2):91--103, 2021.

\bibitem[Tur52]{Turing1952The}
Alan~M. Turing.
\newblock The chemical basis of morphogenesis.
\newblock {\em Phil. Trans. R. Soc. Lond. B}, 237:37--72, 1952.

\bibitem[VHZBG17]{VidalHenriquez2017Convective}
Estefania Vidal-Henriquez, Vladimir Zykov, Eberhard Bodenschatz, and Azam
  Gholami.
\newblock Convective instability and boundary driven oscillations in a
  reaction-diffusion-advection model.
\newblock {\em Chaos}, 27(10):103110, 2017.

\bibitem[VS15]{Volkening2015Modelling}
Alexandria Volkening and Bj{\"o}rn Sandstede.
\newblock Modelling stripe formation in zebrafish: an agent-based approach.
\newblock {\em Journal of the Royal Society Interface}, 12(112):20150812, 2015.

\bibitem[VY21]{Vauchelet2021Numerical}
Nicolas Vauchelet and Shugo Yasuda.
\newblock Numerical scheme for kinetic transport equation with internal state.
\newblock {\em SIAM Journal on Multiscale Modeling and Simulation},
  19(1):184--207, 2021.

\bibitem[WCA{\etalchar{+}}07]{Williams2007Helicobacter}
Susan~Mary Williams, Yu-Ting Chen, Tessa~M Andermann, James~Elliot Carter,
  David~J. McGee, and Karen~M. Ottemann.
\newblock Helicobacter pylori chemotaxis modulates inflammation and
  bacterium-gastric epithelium interactions in infected mice.
\newblock {\em Infection and Immunity}, 75(8):3747--3757, 2007.

\bibitem[WOL{\etalchar{+}}17]{Wang2017A}
Qixuan Wang, Ji~Won Oh, Hye-Lim Lee, Anukriti Dhar, Tao Peng, Raul Ramos,
  Christian~F. Guerrero-Juarez, Xiaojie Wang, Ran Zhao, Xiaoling Cao, Jonathan
  Le, Melisa~A. Fuentes, Shelby~C. Jocoy, Antoni~R. Rossi, Brian Vu, Kim Pham,
  Xiaoyang Wang, Nanda~Maya Mali, Jung~Min Park, June-Hyug Choi, Hyunsu Lee,
  Julien M.~D. Legrand, Eve Kandyba, Jung~Chul Kim, Moon~Kyu Kim, John Foley,
  Zhengquan Yu, Krzysztof Kobielak, Bogi Andersen, Kiarash Khosrotehrani, Qing
  Nie, and Maksim~V. Plikus.
\newblock A multi-scale model for hair follicles reveals heterogeneous domains
  driving rapid spatiotemporal hair growth patterning.
\newblock {\em eLife}, 6, 2017.

\bibitem[XO09]{Xue2009Multiscale}
Chuan Xue and Hans~G. Othmer.
\newblock Multiscale models of taxis-driven patterning in bacterial
  populations.
\newblock {\em SIAM J. Appl. Math.}, 70(1):133--167, 2009.

\bibitem[Xue15]{Xue2015Macroscopic}
Chuan Xue.
\newblock Macroscopic equations for bacterial chemotaxis: integration of
  detailed biochemistry of cell signaling.
\newblock {\em J. Math. Biol.}, 70(1-2):1--44, 2015.

\bibitem[XXT18]{Xue2018The}
Xiaoru Xue, Chuan Xue, and Min Tang.
\newblock The role of intracellular signaling in the stripe formation in
  engineered escherichia coli populations.
\newblock {\em PLoS Computational Biology}, 14(6):e1006178, 2018.

\bibitem[ZW15]{Zhao2015A}
Yonggang Zhao and Mingxin Wang.
\newblock A reaction–diffusion–advection equation with mixed and free
  boundary conditions.
\newblock {\em J. Dynam. Differential Equations}, 30:743--777, 2015.

\end{thebibliography}

\end{document}